\documentclass{conm-p-l}%[12pt]{amsart} 
\usepackage{amssymb,amscd,epsf}

\usepackage[mathscr]{eucal}
\newtheorem{thm}{Theorem}[section]
\newtheorem{lem}[thm]{Lemma}
\newtheorem{definition}[thm]{Definition}
\newtheorem*{main}{Main Theorem}
\numberwithin{equation}{section}
\begin{document}
\title[Crooked tilings]{Affine Schottky Groups and Crooked Tilings}
\author[Charette and Goldman]
{Virginie Charette and William M. Goldman \vspace{-0.3in}}
\address{
Department of Mathematics, University of Maryland,
College Park, MD 20742  
}
\email{
virginie@math.umd.edu, wmg@math.umd.edu}
\date{\today}

\renewcommand{\P}{{\mathbb P}}
\newcommand{\E}{{\mathbb E}}
\newcommand{\Cc}{\mathcal C}  % crooked plane
\newcommand{\Hh}{\mathcal H}  % crooked half-space bounded by crooked plane
\newcommand{\HH}{\mathfrak H}  % crooked half-space in sequence
\newcommand{\zz}{\mathcal Z}  % zigzag
\newcommand{\R}{{\mathbb R}}
\newcommand{\Z}{{\mathbb Z}}
\newcommand{\nn}{{\partial\Nn_+}}
\renewcommand{\L}{{\mathbb L}}
\newcommand{\la}{\langle}
\newcommand{\ra}{\rangle}
\newcommand{\bX}{\Bar{X}}
\newcommand{\bD}{\Bar{\Delta}}
\renewcommand{\o}{\operatorname}
\newcommand{\hyp}{\o{hyp}}
\newcommand{\rto}{\R^{2,1}}
\newcommand{\rno}{\R^{n,1}}
\newcommand{\rn}{\R^n}
\newcommand{\rt}{\R^2}
\newcommand{\rht}{\o{H}^2_\R}
\newcommand{\g}{{\mathfrak g}}
\newcommand{\gl}{{\mathfrak{gl}}}
\newcommand{\Nn}{{\mathcal N}}
\newcommand{\SOto}{\o{SO}(2,1)}
\newcommand{\SOt}{\o{SO}(2)}
\newcommand{\Oto}{\o{O}(2,1)}
\newcommand{\SOoo}{\o{SO}^0(2,1)}
\newcommand{\SOooo}{\o{SO}^0(1,1)}
\newcommand{\rk}{\o{rank}}
\newcommand{\Map}{\o{Map}}
\newcommand{\vx}{{\mathsf{x}}}	%vector variables will be in sans-sserif 
\newcommand{\vy}{{\mathsf{y}}}
\newcommand{\vz}{{\mathsf{z}}}
\newcommand{\vu}{{\mathsf{u}}}
\newcommand{\vv}{{\mathsf{v}}}
\newcommand{\vw}{{\mathsf{w}}}
\newcommand{\va}{{\mathsf{a}}}
\newcommand{\origin}{{\mathsf{0}}} %this is the origin in E; 
%				     O is the ``origin'' in hyperbolic space
\newcommand{\B}{\mathbb B }
\newcommand{\Id}{\mathbb I}
\newcommand{\xo}[1]{{\vx}^0(#1)}
\newcommand{\xp}[1]{{\vx}^+(#1)}
\newcommand{\xm}[1]{{\vx}^-(#1)}
\newcommand{\xpm}[1]{{\vx}^{\pm}(#1)}
\newcommand{\xmp}[1]{{\vx}^{\mp}(#1)}
\newcommand{\Stem}{{\mathcal{S}}}
\newcommand{\Ss}{{\mathfrak{S}}}
\newcommand{\Wing}{{\mathbf{Wing}}}
\newcommand{\nWing}{{\mathbf{nWing}}}
\newcommand{\Wingp}{{\mathcal W}^+}
\newcommand{\Wingm}{{\mathcal W}^-}
\newcommand{\Wingpm}{{\mathcal W}^\pm}
\newcommand{\nWingp}{{\mathcal M}^+}
\newcommand{\nWingm}{{\mathcal M}^-}
\newcommand{\nCP}{{\mathcal K}}
\newcommand{\csch}{\o{csch}}
\newcommand{\sech}{\o{sech}}
\newcommand{\Isom}{\o{Isom}}
\newcommand{\Isomo}{\Isom^0}
\newcommand{\is}{\Isom(\E)}
\newcommand{\iso}{\Isomo(\E)}
\thanks{The authors gratefully acknowledge partial support from NSF grant
DMS-9803518. Goldman gratefully acknowledges a Semester Research Award from 
the General Research Board of the University of Maryland.}
% \begin{abstract}
% This paper gives an exposition of the following result
% of Drumm. Let $\Cc_1^+,\Cc_1^-,\dots,\Cc_m^+,\Cc_m^-$ be a family
% of disjoint crooked planes in Minkowski $(2+1)$-space $\E$ which are paired
% by affine isometries $h_i:\Cc_i^-\longrightarrow \Cc_i^+$. 
% Then $h_1,\dots,h_m$ freely generate a group $\Gamma$ which acts properly
% discontinuously on $\E$ with fundamental domain bounded by
% $\Cc_1^+,\Cc_1^-,\dots,\Cc_m^+,\Cc_m^-$.
% \end{abstract}
\maketitle 
\tableofcontents 

In his doctoral thesis~\cite{Drumm0} and subsequent
papers~\cite{Drumm1,Drumm2}, Todd Drumm developed a theory of
fundamental domains for discrete groups of isometries of Minkowski
$(2+1)$-space $\E$, using polyhedra called {\em crooked planes\/} 
and {\em crooked half-spaces.\/} This paper expounds these results. 
\begin{main}
Let $\Cc_1^-,\Cc_1^+,\dots,\Cc_m^-,\Cc_m^+\subset\E$ be a family of
crooked planes bounding crooked half-spaces $\Hh_1^-,\Hh_1^+,
\dots,\Hh_m^-,\Hh_m^+\subset\E$ and 
$h_1,\dots, h_m\in\iso$ such that:
\begin{itemize}
\item any pair of the $\Hh_i^j$ are pairwise disjoint;
\item $h_i(\Hh_i^-)= \E - \Bar{\Hh}_i^+$. 
\end{itemize}
Then:
\begin{itemize}
\item the group $\Gamma$ generated by $h_1,\dots, h_m\in\iso$ acts properly
discontinuously on $\E$;
\item the polyhedron
\begin{equation*}
X = \E - \bigcup_{i=1}^m (\Bar{\Hh}_i^+ \cup \Bar{\Hh}_i^-)
\end{equation*}
is a fundamental domain for the $\Gamma$-action on $\E$.
\end{itemize}
\end{main}
The first examples of properly discontinuous affine actions of
nonabelian free groups were constructed by
Margulis~\cite{Margulis1,Margulis2}, following a suggestion of
Milnor~\cite{Milnor}. (Compact quotients of affine space were
classified by Fried and Goldman~\cite{FriedGoldman} about the same time.)
For background on this problem, we refer the reader to these articles
as well as the survey articles \cite{CDGM},\cite{Drumm4}.

Here is the outline of this paper. In \S\ref{sec:Minkowski}, we
collect basic facts about the geometry of Minkowski space and $\rto$.
 In \S\ref{sec:compression}, we prove a basic lemma on how
isometries compress Euclidean balls in special directions. A key idea
is the {\em hyperbolicity\/} of a hyperbolic isometry, motivated by
ideas of Margulis~\cite{Margulis1,Margulis2}, and discussed in
\S\ref{sec:hyperbolicity}, using the {\em null frames\/} associated
with spacelike vectors and hyperbolic isometries. The hyperbolicity
of a hyperbolic element $g$ is defined as the distance between the attracting
and repelling directions, and $g$ is $\epsilon$-hyperbolic if its hyperbolicity
is $\ge \epsilon$.

\S\ref{sec:Schottky} reviews Schottky subgroups of $\SOoo$ acting on
$\rht$. This both serves as the prototype for the subsequent
discussion of affine Schottky groups and as the starting point for the
construction of affine Schottky groups. For Schottky groups acting on
$\rht$, a completeness argument in Poincar\'e's polygon theorem shows
that the images of the fundamental polygon tile all of $\rht$; the
analogous result for affine Schottky groups is the main topic of this
paper.

\S\ref{sec:hypcrit} gives a criterion for
$\epsilon$-hyperbolicity of elements of Schottky groups.

\S\ref{sec:crooked} introduces {\em crooked planes\/} as the analogs 
of geodesics in $\rht$ bounding the fundamental polygon.
\S\ref{sec:extSchottky1}
extends Schottky groups to linear actions on
$\rto$. In \S\ref{sec:affineSchottky}, 
affine deformations of these linear actions are
proved to be properly discontinuous on open subsets of Minkowski
space, using the standard argument for Schottky groups. 
The fundamental polyhedron $X$ is bounded by crooked planes, in exactly
the same configuration as the geodesics bounding the fundamental polygon
for Schottky groups. The generators of the affine Schottky group pair
the faces of $X$ in exactly the same pattern as for the original Schottky 
group.

The difficult part of the proof is to show that the images $\gamma\bX$ tile
{\em all\/} of Minkowski space. Assuming that a point $p$ lies in $\E-\Gamma\bX$, Drumm
intersects the tiling with a fixed definite plane $P\ni p$. 
In \S\ref{sec:zigzags}, we abstract this idea by introducing 
{\em zigzags} and {\em zigzag regions,\/}
which are the intersections of crooked planes and half-spaces with $P$. 

The proof that $\Gamma\bX=\E$ (that is, {\em completeness\/} of the
affine structure on the quotient $(\Gamma\bX)/\Gamma$) involves a
nested sequence of crooked half-spaces $\HH_k$ containing $p$.
This sequence is constructed in \S\ref{sec:sequence} and 
is indexed by a sequence $\gamma_k\in\Gamma$.
\S\ref{sec:uniform} gives a uniform lower bound to the Euclidean width
of $\bX$. Bounding the uniform width is a key ingredient in proving
completeness for Schottky groups in hyperbolic space. 

\S\ref{sec:approx} approximates the nested sequence of zigzag regions
by a nested sequence of half-planes $\Pi_k$ in $P$. 
The uniform width bounds are used to prove Lemma~\ref{lem:tubnbhd}, 
on the existence of tubular neighborhoods $T_k$ of $\partial\Pi_k$ 
which are disjoint. The Compression Lemma~\ref{lem:affinecompression},
combined with the observation (Lemma~\ref{lem:wunu}) that the zigzag regions
do not point too far away from the direction of expansion of the $\gamma_k$,
imply that the Euclidean width of the $T_k$ are bounded below in terms of
the hyperbolicity of $\gamma_k$. (Lemma~\ref{lem:separation}). 
Thus it suffices to find $\epsilon>0$ such that infinitely many $\gamma_k$ are
$\epsilon$-hyperbolic.

\S\ref{sec:alternative} applies the criterion
for $\epsilon$-hyperbolicity derived in \S\ref{sec:hypcrit} (Lemma~\ref{lem:fxpts}) to
find $\epsilon_0>0$ guaranteeing $\epsilon_0$-hyperbolicity in many cases.
In the other cases, the sequence $\gamma_k$ has a special form, the analysis
of which gives a smaller $\epsilon$ such that every $\gamma_k$ is now
$\epsilon$-hyperbolic. The details of this constitute \S\ref{sec:change}.

We follow Drumm's proof, with several modifications.  We wish to
thank Todd Drumm for the inspiration for this work and Maria Morrill,
as well as Todd Drumm, for several helpful conversations. 

\section{Minkowski space}\label{sec:Minkowski}
We begin with background on $\rto$ and Minkowski (2+1)-space $\E$.
$\rto$ is defined as a real inner product space of dimension 3
with a nondegenerate inner product of index 1.
Minkowski space is an affine space $\E$ whose underlying
vector space is $\rto$; equivalently $\E$ is a simply-connected
geodesically complete flat Lorentzian manifold.
If $p,q\in\E$, then a unique vector $\vv\in\rto$ represents the
{\em displacement\/} from $p$ to $q$, that is, translation by the vector $\vv$ 
is the unique translation taking $p$ to $q$; we write
\begin{equation*}
\vv = q - p \text{~and~}  q = p + \vv.
\end{equation*}

Lines and planes in $\rto$ are classified in terms
of the inner product. The identity component $\SOoo$ of the group of
linear isometries of $\rto$ comprises linear
transformations preserving the {\em future\/} component $\Nn_+$
of the set $\Nn$ of timelike vectors, as well as orientation. The set of rays 
in $\Nn_+$ is
a model for {\em hyperbolic 2-space\/} $\rht$, and the geometry of
$\rht$ serves as a model for the geometry of $\rto$ and $\E$.

A key role is played by the {\em ideal boundary\/} of $\rht$.  Its
intrinsic description is as the projectivization of the lightcone in
$\rto$, although following Margulis~\cite{Margulis1,Margulis2} we
identify it with a section on a Euclidean
sphere. However, to simplify the formulas, we find it more convenient
to take for this section the intersection of the future-pointing
lightcone with the sphere $S^2(\sqrt{2})$ of radius $\sqrt{2}$, rather
than the Euclidean unit sphere, which is used in the earlier
literature. We hope this departure from tradition justifies itself by
the resulting % clarification of ideas.
simplification of notation. %correction

\subsection{Minkowski space and its projectivization}

Let $\rto$ be the three-dimensional real vector space $\R^3$ with the inner
product
\begin{equation*}
\B(\vu,\vv) = \vu_1\vv_1  + \vu_2\vv_2  - \vu_3\vv_3
\end{equation*}
where
\begin{equation*}
\vu = \bmatrix  \vu_1 \\ \vu_2 \\ \vu_3 \endbmatrix,
\vv = \bmatrix  \vv_1 \\ \vv_2 \\ \vv_3 \endbmatrix \in\rto.
\end{equation*}

\begin{definition}
A vector $\vu$ is said to be 
\begin{itemize}
\item {\em spacelike} if $\B(\vu,\vu) >0$,
\item {\em lightlike} (or {\em null}) if $\B(\vu,\vu)=0$,
\item {\em timelike} if $\B(\vu,\vu)<0$.
\end{itemize}
The union of all null lines is called the {\em lightcone}.
\end{definition}
Here is the dual trichotomy for planes in $\rto$:
\begin{definition}
A 2-plane $P\subset\rto$ is: 
\begin{itemize}
\item 
{\em indefinite\/} if the restriction of
$\B$ to $P$ is indefinite, or equivalently if $P\cap\mathcal N\neq\emptyset$; 
\item 
{\em null\/} if
$P\cap\Bar{\Nn}$ is a null line; 
\item {\em definite\/} if the
restriction $\B|_P$ is positive definite, or equivalently if
$P\cap\Bar{\Nn}=\{{\origin}\}$.
\end{itemize}
\end{definition}

Let $\Oto$ denote the group of linear isometries of $\rto$ and $\SOto
=\Oto \cap \o{SL}(3,\R)$ as usual.  Let $\SOoo$ denote the identity
component of $\Oto$ (or $\SOto$).  The affine space with underlying
vector space $\rto$ has a complete flat Lorentzian metric arising from
the inner product on $\rto$; we call it {\em Minkowski space\/} and
denote it by $\E$.  Its isometry group $\is$ of $\E$ splits as a
semidirect product $\Oto\ltimes V$ (where $V$ denotes the vector group
of translations of $\E$) and its identity component $\iso$ of $\is$ is
$\SOoo\ltimes V$. The elements of $\is$ are called {\em affine
isometries,\/} to distinguish them from the linear isometries in
$\Oto$. 

The linear isometries can be characterized as the affine
isometries fixing a chosen ``origin'', which is used to
identify the {\em points\/} of $\E$ with the {\em vectors\/} in $\rto$.
The {\em origin\/} is the point which corresponds to the zero vector
$\origin\in\rto$.
Let $\L:\iso\longrightarrow\SOoo$ denote the
homomorphism associating to an affine isometry its linear part.

Let $\P(\rto)$ denote the {\em projective space\/} associated to $\rto$,
that is the space of 1-dimensional linear subspaces of $\rto$. Let
\begin{equation*}
\P:\rto-\{{\origin}\}\longrightarrow\P(\rto) 
\end{equation*}
denote the quotient mapping which associates to a nonzero vector the
line it spans.

Let $\Nn$ denote the set of all timelike vectors.
Its two connected components are the convex cones
\begin{equation*}
\Nn_+ = \{\vv\in\Nn \mid \vv_3 > 0\}
\end{equation*}
and
\begin{equation*}
\Nn_- = \{\vv\in\Nn \mid \vv_3 < 0\}.
\end{equation*}
We call $\Nn_+$  the {\em future\/} or {\em positive time-orientation\/} 
of ${\origin}$.

{\em Hyperbolic 2-space\/} $\rht=\P(\Nn)$ consists of timelike lines
in $\rto$.  The distance $d(u,v)$ between two points 
$u = \P(\vu), v = \P(\vv)$ represented
by timelike vectors $\vu,\vv\in\Nn$ is given by:
\begin{equation*}
\cosh(d(u,v)) = \frac{\vert\B(\vu,\vv)\vert}{\sqrt{\B(\vu,\vu)\B(\vv,\vv)}}.
\end{equation*}
The hyperbolic plane can be identified with either one of the
components of the two-sheeted hyperboloid defined by $\B(\vv,\vv) = -1$.
This hyperboloid  is a section of 
$\P:\Nn\longrightarrow \rht$, and inherits a complete Riemannian metric
of constant curvature $-1$ from $\rto$.

The group $\SOoo$ acts linearly on $\Nn$, and thus by projective
transformations on $\rht$. It preserves the subsets $\Nn_\pm$ and acts
isometrically with respect to the induced Riemannian structures.

The boundary $\partial\Nn$ is the union of all null lines, that is
the lightcone. The projectivization of $\partial\Nn -\{ {\origin}\}$
is the {\em ideal boundary\/} $\partial\rht$ of hyperbolic 2-space. 

Let $\Ss$ denote the set of unit-spacelike vectors : 
\begin{equation*}
\Ss = \{\vv\in\rto \mid \B(\vv,\vv) = 1\} .
\end{equation*}
It is a one-sheeted hyperboloid which is homeomorphic to an annulus.
% With its induced Lorentzian structure $\Ss$ is 
% {\em de Sitter space,\/} a complete Lorentzian manifold of constant
% curvature +1. 
% Its two rulings by null lines form the lightcones of this induced
% Lorentzian structure. Its image in the projective plane 
% $\P(\Ss)\subset \P(\rto)$ is {\em projective de Sitter space,\/}
% and is homeomorphic to a M\"obius band. 
Points in $\Ss$ correspond to oriented geodesics in $\rht$, or geodesic
half-planes in $\rht$ as follows. Let $\vv\in\Ss$. Then 
define the {\sl half-plane\/} $H_\vv = \P(\tilde{H}_\vv)$ where
\begin{equation*}
\tilde{H}_\vv = \{\vu\in\Nn_+ \mid \B(\vu,\vv) > 0\}.
\end{equation*}
$H_\vv$ is one of the two components of the complement 
$\rht - l_\vv$ where
\begin{equation}\label{eq:spaceline}
l_\vv  = \P(\{\vu\in\Nn_+ \mid \B(\vu,\vv) = 0\}) = \partial H_\vv 
\end{equation}
is the geodesic in $\rht$ corresponding to the line $\R\vv$ spanned by $\vv$.

\subsection{A little Euclidean geometry}\label{sec:Euclid}
%\begin{definition}
Denote the Euclidean norm by
\begin{equation*}
\Vert\vv\Vert = \sqrt{(\vv_1)^2 + (\vv_2)^2 + (\vv_3)^2}
\end{equation*}
 and let $\rho$ denote the Euclidean distance
on $\E$ defined by
\begin{equation*}
\rho(p,q) = \Vert p - q\Vert .
\end{equation*}
If $S\subset\E$ and $\delta>0$, the 
{\em Euclidean $\delta$-neighborhood of $S$\/} is
$B(S,\delta) = \{y\in \E \mid \rho(S,y) < \delta\}$.
Note that $B(S,\delta) = \bigcup_{x\in S}B(x,\delta)$.
%\end{definition}

Let
\begin{equation*}
S^2(\sqrt{2}) = \{\vv\in\rto \mid \Vert\vv\Vert = \sqrt{2} \}
\end{equation*}
be the Euclidean sphere of radius $\sqrt{2}$.
Let $S^1$ denote the intersection $S^2(\sqrt{2})\cap\partial\Nn_+$, 
consisting of points
\begin{equation*}
\vu_\phi = \bmatrix \cos(\phi) \\ \sin(\phi) \\ 1 \endbmatrix %.
\end{equation*}
where $\phi\in\R$. %correction
The subgroup of $\SOoo$ preserving $S^1$ and $S^2(\sqrt{2})$ is the
subgroup $\SOt$ consisting of rotations
\begin{equation*}
R_\phi = \bmatrix  
\cos(\phi) &  -\sin(\phi) & 0 \\
\sin(\phi) &   \cos(\phi) & 0 \\
0 	&   0 	& 1 \endbmatrix.
\end{equation*}
% We shall take for $S^1$ a model for the projectivization $\P(\nn)
% \approx \partial\rht$. 
While the linear action of $\SOoo$ does not preserve 
$S^1$, we may use the identification of $S^1$ with 
$\P(\nn)$ to define an action %. 
of $\SOoo$ on $S^1$. %correction
If $g\in\SOoo$, we denote this action
by $\P(g)$, that is, if $\vu\in S^1$, then
$\P(g)(\vu)$ is the image of $\P(g(\vu))$ under the identification
$\P(\nn)\longrightarrow S^1$.
Throughout this paper (and especially in \S 2), we 
shall consider this action of $\SOoo$ on $S^1$.

The restriction of either the Euclidean metric or the Lorentzian
metric to $S^1$ is the rotationally invariant metric $d\phi^2$ on
$S^1$ for which the total circumference is $2\pi$.

\begin{definition}
An {\em interval\/} on $S^1$  is an open subset $A$ of the form
$\{\vu_\phi\mid \phi_1<\phi<\phi_2\}$ where $\phi_1 < \phi_2$ and
$\phi_2 - \phi_1 < 2\pi$. Its {\em length\/} $\Phi(A)$ is $\phi_2-\phi_1$.
\end{definition}

Note that if $\phi_1<\phi_2$, the points $\va_1=\vu_{\phi_1}$ and $\va_2=\vu_{\phi_2}\in S^1$ bound two
different intervals : we can either take $A=\{\vu_\phi\mid
\phi_1<\phi<\phi_2\}$ or 
$A=\{\vu_\phi\mid \phi_2<\phi<\phi_1+2\pi\}$.  The length of one of
these intervals is less than or equal to $\pi$, in which case
\begin{equation*}
\rho(\va_1,\va_2)  = 2\sin(\Phi(A)/2).
\end{equation*}

Intervals correspond to unit-spacelike vectors as follows.  Let
$A\subset S^1$ be an interval bounded by $\va_1=\vu_{\phi_1}$ and
$\va_2=\vu_{\phi_2}$, where $\phi_1 <\phi_2$.  Then the Lorentzian cross-product $\vu_2\boxtimes\vu_1$
(see, for example, Drumm-Goldman~\cite{DrummGoldman1,DrummGoldman2})
is a positive scalar multiple of the corresponding unit-spacelike vector.

In \cite{Drumm0,Drumm1,Drumm2}, Drumm considers conical neighborhoods
in $\nn$ rather than intervals in $S^1$. 
A {\em conical neighborhood\/}  is a connected open subset $U$
of the future lightcone $\nn$ which is 
invariant under the group $\R_+$ of positive scalar multiplications.
The projectivization $\P(U)$ of a conical neighborhood
is a connected open interval in $\P(\nn)\approx S^1$ which we may 
identify with the interval $U\cap S^1$, which is an interval.
Thus every conical neighborhood  equals $\R_+\cdot A$, where $A\subset S^1$
is the interval $A = U\cap S^1$. 
% The endpoints of an arc $A$ in 
% $S^2 \cap \nn$ are the endpoints of a unique geodesic 
% $\gamma\in\rht$, which bounds two {\em geodesic 
% half-planes\/} in $\rht$, the one whose contains $A$ is its {\em
% convex hull\/} $H_A$. We denote the corresponding polar vector by
% $\vv_A$.

\subsection{Null frames}\label{sec:nullframes}
Let $\vv\in\Ss$ be a unit-spacelike vector. We associate to $\vv$ 
two null vectors $\xpm{\vv}$ in the future which are
$\B$-orthogonal to $\vv$ and % have Euclidean length 1.
lie on the unit circle $S^1=S^2(\sqrt{2})\cap\nn$.
These vectors
correspond to the endpoints of the geodesic $l_\vv$. To define $\xpm{\vv}$,
first observe that the orthogonal complement
\begin{equation*}
\vv^\perp = \{\vu\in\rto\mid  \B(\vu,\vv) = 0\} 
\end{equation*}
meets the positive lightcone $\nn$ in two rays.
Then $S^1 = \nn\cap S^2(\sqrt{2})$
meets $\vv^\perp$ in a pair of vectors $\xpm{\vv}$. We determine which one
of this pair is $\xp{\vv}$ and which one is $\xm{\vv}$ by requiring that the
triple
\begin{equation*}
(\xm{\vv},\xp{\vv},\vv) 
\end{equation*}
be a positively oriented basis of $\rto$.
We call such a basis a {\em null frame\/} of $\rto$.  
(Compare Figure~\ref{fig:frame}.) The pair
$\{\xm{\vv},\xp{\vv}\}$ is a basis for the indefinite plane
$\vv^\perp$.  In fact, $\vv$ is the
unit-spacelike vector corresponding to the interval bounded by the
ordered pair $(\xp{\vv},\xm{\vv})$.

Hyperbolic elements of $\SOoo$ also determine null frames.  
Recall that $g\in\SOoo$ is {\em hyperbolic\/} if
it has real distinct eigenvalues, which are necessarily positive. 
Then the eigenvalues are $\lambda,1,\lambda^{-1}$ and we may assume that
\begin{equation*}
\lambda < 1 < \lambda^{-1}.  
\end{equation*}
Let $\xm{g}$ denote the unique eigenvector with eigenvalue $\lambda$ lying on
$S^1$ and $\xp{g}$ denote the unique eigenvector with eigenvalue 
$\lambda^{-1}$ lying on $S^1$. Then $\xo{g}\in\Ss$ is the uniquely 
determined eigenvector of $g$ such that $(\xm{g},\xp{g},\xo{g})$ is positively
oriented.  Observe that this is a null frame, since
$\xpm{g}=\xpm{\vv}$, where $\vv =\xo{g}$.  

\subsection{$\epsilon$-Hyperbolicity}\label{sec:hyperbolicity}
We may define the {\em hyperbolicity\/} of a unit-spacelike vector $\vv$ 
as the Euclidean distance
$\rho(\xp{\vv},\xm{\vv})$. 
The following definition (Drumm-Goldman~\cite{DrummGoldman1}) is
based on Margulis~\cite{Margulis1,Margulis2}.
\begin{definition}
A unit-spacelike vector $\vv$ is {\em $\epsilon$-spacelike} if
$ \rho(\xp{\vv},\xm{\vv}) \ge \epsilon$.
A hyperbolic element $g\in\SOoo$ is {\em $\epsilon$-hyperbolic} if
$\xo{g}$ is $\epsilon$-spacelike.
An affine isometry is {\em $\epsilon$-hyperbolic} if its linear part
is an $\epsilon$-hyperbolic linear isometry.
\end{definition}
The spacelike vector $\vv$ corresponds to a geodesic 
$l_\vv$ (defined in \eqref{eq:spaceline})
in the hyperbolic plane $\rht$. 
Let 
\begin{equation}\label{eq:rhtorigin}
O = \P\left(\bmatrix 0 \\ 0 \\ 1 \endbmatrix\right) 
\end{equation}
be the {\em origin\/} in $\rht$. 
Although we will not need this, the hyperbolicity relates to
other more familiar quantities. For example, the hyperbolicity of
a vector $\vv$ relates to the distance from $l_\vv$ to the origin
$O$ in $\rht$ and to the Euclidean length of $\vv$ by:
\begin{align*}
\rho(\xp{\vv},\xm{\vv}) & = 2  \sech (d(O,l_\vv)) \\
& = 2 \sqrt{\frac2{1 + \Vert\vv\Vert^2}}. 
\end{align*}
The set of all $\epsilon$-spacelike vectors is the compact set
\begin{equation*}
\Ss_\epsilon =  \Ss \cap B({\origin},\sqrt{8/\epsilon^2 - 1})
\end{equation*}
and $\Ss = \bigcup_{\epsilon>0}\Ss_\epsilon$.
%%%%%%%%%%%%%%%%%%%%%%%%%%%%%%%%%%%%%%%%%%%%%%%%%%%%%%%%%%%%%%%%%%%%%%%%%%%%%%
% do we want more details of the proof here? If d = d(O,l_\vv)), then
% take \vv = \bmatrix 0 \\ \cosh(d)  \\ sinh(d)\endbmatrix
% so that xp = \bmatrix -\sech(d) \\ \tanh(d) \\1  \endbmatrix
% and     xm = \bmatrix  \sech(d) \\ \tanh(d) \\1  \endbmatrix
% Then \rho(xp,xm) = 2 \sech(d)\
% and   \Vert\vv\Vert = \cosh^2(d) + \sinh^2(d) = 2\cosh^2(d) -1 etc. etc.
%%%%%%%%%%%%%%%%%%%%%%%%%%%%%%%%%%%%%%%%%%%%%%%%%%%%%%%%%%%%%%%%%%%%%%%%%%%%%%

% \subsection{The boundary of de Sitter space}
% If $\vv_n\in\Ss$ is a sequence which leaves every compact set in
% $\Ss$, then by \eqref{eq:hypElength} again,
% \begin{equation*}
% \hyp(\vv_n)\longrightarrow 0.  
% \end{equation*}
% Furthermore, a convergent sequence $\P(\vv_n)$ in $\P(\rto)$ 
% (the projectivization of a null line). We thus compactify $\Ss$ by adding
% the projectivized lightcone $\P(\partial\Nn)$. The resulting compactification
% \begin{equation*}
% \Bar{\Ss} = \Ss  \cup \P(\partial\Nn)
% \end{equation*}
% is homeomorphic to a Klein bottle, which is just the closure of
% $\Ss\subset\rto$ in real projective 3-space $\P(\rto\oplus\R)\approx\R\P^3$.
 
\begin{figure}[ht]
\centerline{\epsfxsize=3in \epsfbox{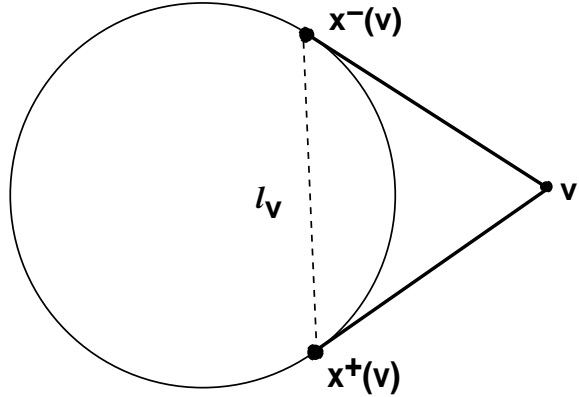}}
\caption
{A null frame}
\label{fig:frame}
\end{figure}
\begin{figure}[ht]
\centerline{\epsfxsize=3in \epsfbox{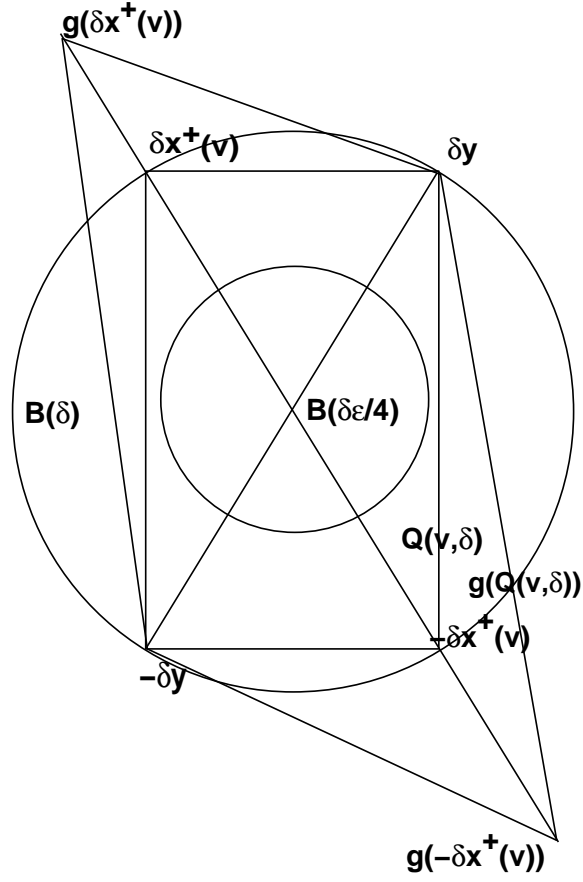}}
\caption{The Compression Lemma}
\label{fig:compression}
\end{figure}

\subsection{Compression}\label{sec:compression}
We now prove the basic technical lemma bounding the compression of
Euclidean balls by isometries of $\E$.
A similar bound was found in Lemma 3 (pp.\ 686--687) of
Drumm~\cite{Drumm2} (see also (1) of \S3.7 of Drumm~\cite{Drumm1}).

\begin{definition}\label{def:weakunstable}
For any unit-spacelike vector $\vv$, the {\em weak-unstable
plane\/} $E^{wu}(\vv)\subset\rto$ is the plane spanned by the vectors
$\vv$ and $\xp{\vv}$. If $g\in\SOto$ is hyperbolic, then $E^{wu}(g)$
is defined as $E^{wu}(\xo{g})$. 
If $x\in\E$, then
$E^{wu}_x(g)$ is defined as the image of $E^{wu}(g)$ under translation 
by $x$.
\end{definition}
\noindent (Drumm~\cite{Drumm0,Drumm1,Drumm2} denotes $E^{wu}$ by $S_+$.)

\begin{lem}\label{lem:inscribedball}
Let $\vv$ be an $\epsilon$-spacelike vector.
Let $\delta>0$ and $Q(\vv,\delta)$ be the convex hull of the four intersection
points 
of $\partial B({\origin},\delta)$ with the two lines
$\R\vv$ and $\R\xp{\vv}$. Then $Q(\vv,\delta)$ is a rectangle
in the plane $E^{wu}(\vv)$ containing
\begin{equation*}
B\left(\origin,\frac{\delta\epsilon}4\right)
\cap E^{wu}(\vv).
\end{equation*}
\end{lem}
\begin{proof}
The lines $\R\vv$ and $\R\xp{\vv}$ meet 
$\partial B({\origin},\delta)$ in the four points
$\pm \delta \xp{\vv}, \pm \delta \vy$
where $\vy= \vv/\Vert \vv\Vert$.
(Compare Figure~\ref{fig:compression}.)
Since $\triangle(\xp{\vv},\vy,\xm{\vv})$ is isosceles,
\begin{equation*}
\rho(\xp{\vv},\vy) = \rho(\xm{\vv},\vy),
\end{equation*}
the triangle inequality implies
\begin{align*}
\epsilon & \le \rho(\xp{\vv},\xm{\vv})\\   & \le
\rho(\xp{\vv},\vy) + \rho(\vy,\xm{\vv}) \\ & \le 2 \rho(\xp{\vv},\vy).
\end{align*}
Therefore $\rho(\xp{\vv},\vy) \ge \epsilon/2$.
Similarly $\rho(\xp{\vv},-\vy) \ge \epsilon/2.$
Thus the sides
of $Q(\vv,\delta)$ have length at least $\delta\epsilon/2$,
and $B({\origin},\delta\epsilon/4)\subset Q(\vv,\delta)$ as claimed.
\end{proof}

\begin{lem}\label{lem:CompressionLemma}
Suppose that $g\in\SOoo$ is $\epsilon$-hyperbolic. 
Then for all $\delta>0$,
\begin{equation*}
B\left(\origin,\frac{\delta\epsilon}4\right) \cap E^{wu}(g) \subset
gB({\origin},\delta). 
\end{equation*}
\end{lem}
\begin{proof}
$g$ fixes $\pm\delta\vy$ and multiplies $\pm\delta\xp{g}$ 
by $\lambda(g)^{-1}>1$ so $Q \subset g(Q)$.
(Compare Figure~\ref{fig:compression}.)
By Lemma~\ref{lem:inscribedball},
$B(\origin,\delta\epsilon/4) \cap E^{wu}(g) \subset Q 
\subset gB({\origin},\delta).$

\end{proof}
The following lemma directly follows from Lemma~\ref{lem:CompressionLemma}
by applying translations.
\begin{lem}[The Compression Lemma]\label{lem:affinecompression}
Suppose that $h\in\Isomo(\E)$ is an $\epsilon$-hyperbolic affine 
isometry.
Then for all $\delta>0$ and $x\in\E$,
\begin{equation*}
B\left(h(x),\frac{\delta\epsilon}4\right) \cap E^{wu}_{h(x)}(g) \subset
h(B(x,\delta)). 
\end{equation*}
\end{lem}  

\section{Schottky groups}\label{sec:Schottky}
In this section we recall the construction of Schottky subgroups in
$\SOoo$ and their action on $\rht$. We also use the projective action
$\P(g)$ of $g\in\SOoo$ on $S^1$ (see \S\ref{sec:Euclid}) and
abbreviate $\P(g)$ by $g$ to ease notation.

This classical construction is the template for the construction of
affine Schottky groups later in \S\ref{sec:affineSchottky}.  Then we
prove several elementary technical facts to be used later in the proof 
of the Main Theorem.

\subsection{Schottky's configuration}
Let $G\subset\SOoo$ be a group generated be 
$g_1,\dots,g_m$, for which there exist intervals
$A_i^-,A_i^+\subset S^1$, $i=1,\dots,m$ such that: 
\begin{itemize}
\item $g_i(A_i^-)  = \nn - \Bar{A}_i^+$;  
\item $g_i^{-1}(A_i^+)  = \nn - \Bar{A}_i^-$.  
\end{itemize}
We call $G$ a {\em Schottky group}.  Write $J$ for the set $\{+1,-1\}$ or its abbreviated version
$\{+,-\}$. Denote by $I$ 
the set $\{1,\dots,m\}$. We index many of the objects associated with Schottky
groups (for example the intervals $A_i^j$ and the Schottky
generators $g_i^j$) by the Cartesian product 
\begin{equation*}
I\times J =  \{1,\dots,m\} \times \{+,-\}.
\end{equation*}
Let $H_i^+$ and $H_i^-$ be the two half-spaces
(the convex hulls) in $\rht$ bounded by
$A_i^+$ and $A_i^-$ respectively. 
Their complement 
\begin{equation*}
\Delta_i = \rht-(\Bar{H}_i^+\cup \Bar{H}_i^-) 
\end{equation*}
is the convex hull in $\rht$ of 
$\partial\rht-(\Bar{A}_i^+\cup \Bar{A}_i^-). $
These half-spaces satisfy conditions analogous to those above : 
\begin{itemize}
\item $g_i(H_i^-)  = \rht - \Bar{H}_i^+$;  
\item $g_i^{-1}(H_i^+)  = \rht - \Bar{H}_i^-$.
\end{itemize}
(Compare Fig.~\ref{fig:delta1}.)  $\Delta_i$ is a fundamental domain
for the cyclic group $\la g_i\ra$. As all of the $A_i^j$ are disjoint, 
all of the complements $\rht -\Delta_i$ are disjoint.

\begin{lem}\label{lem:Brouwer}
For each $i=1,\dots,m$, $\xp{g_i}\in A_i^+$ and 
$\xm{g_i}\in A_i^-$.
\end{lem}
\begin{proof}

Since $g_i$ is hyperbolic, it has three invariant lines corresponding
to its eigenvalues. The two eigenvectors corresponding to $\lambda$ and
$\lambda^{-1}$ are null, determining exactly two fixed points of $g_i$
on $S^1$. The fixed point corresponding to $\xp{g_i}$ is attracting 
and the fixed point corresponding to $\xm{g_i}$ is repelling.
Since $A_i^+\subset S^1 - A_i^-$,
\begin{equation*}
g_i(A_i^+) \subset g_i(S^1 - A_i^-) = A_i^+ 
\end{equation*}
so Brouwer's fixed-point theorem implies that
either $\xp{g_i}$ or $\xm{g_i}$ lies in $A_i^+$.
The same argument applied to $g_i^{-1}$ implies that 
either $\xp{g_i}$ or $\xm{g_i}$ lies in $A_i^-$.
Since % there are 
$g$ has %correction
only two fixed points and $A_i^-\cap A_i^+=\emptyset$,
either $\xp{g_i}\in A_i^+, \xm{g_i}\in A_i^-$ or
$\xp{g_i}\in A_i^-, \xm{g_i}\in A_i^+$. The latter case cannot happen
since $\xp{g_i}$ is attracting and
$\xm{g_i}$ is repelling.
\end{proof}

\begin{figure}[ht]
\centerline{\epsfxsize=3in \epsfbox{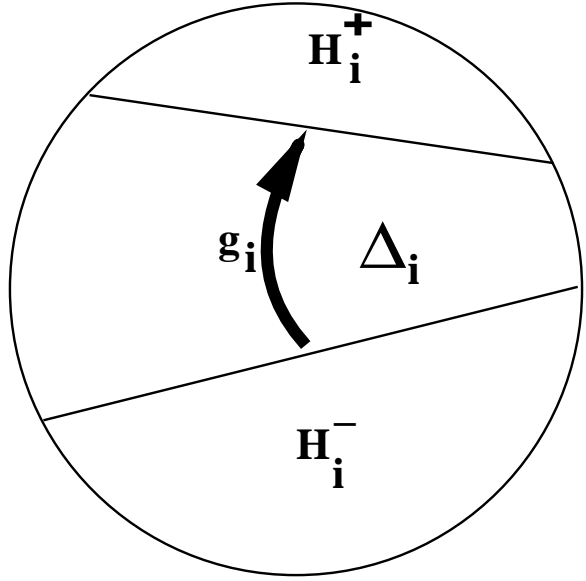}}
\caption
{Half-planes defining a fundamental domain for a cyclic hyperbolic group}
\label{fig:delta1}
\end{figure}

\begin{figure}[ht]
\centerline{\epsfxsize=3in \epsfbox{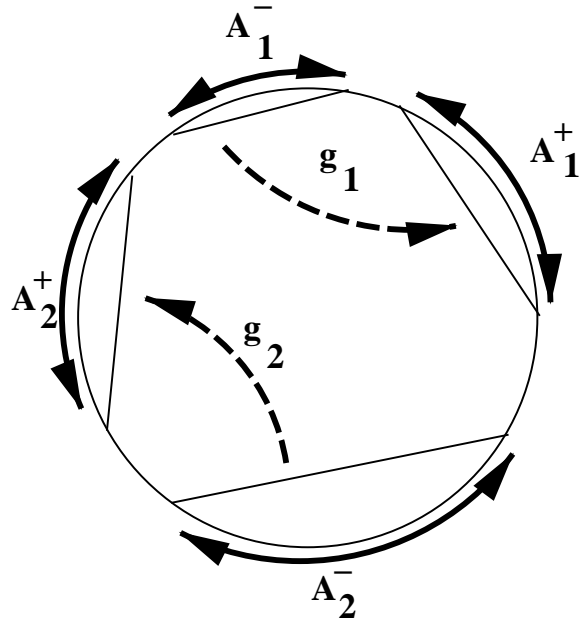}}
\caption
{Generators for a Schottky group}
\label{fig:delta2}
\end{figure}

The following theorem is the basic result on Schottky groups.  It is
one of the simplest cases of ``Poincar\'e's
theorem on fundamental polygons'' or ``Klein's combination theorem.''
Compare Beardon~\cite{Beardon}, Ford~\cite{Ford},
Ratcliffe~\cite{Ratcliffe}, pp. 584--587, and
Epstein-Petronio~\cite{EpsteinPetronio}.

\begin{thm}\label{thm:Schottky}
The set $\{g_1,\dots,g_m\}$ freely generates $G$ and $G$ is discrete.
The intersection $\Delta = \Delta_1\cap \dots \cap \Delta_m$
is a fundamental domain for $G$ acting on $\rht$. 
\end{thm}
\noindent Figure~\ref{fig:delta2} depicts the pattern of identifications.

We break the proof into three separate lemmas: first, that the
$g\Delta$ form a set of disjoint tiles, second, that $G$ is
discrete and third, that these tiles cover all of $\rht$.  The first
lemma extends immediately to affine Schottky groups.  The second lemma
is automatic since the linear part of $\Gamma$ equals $G$, which we
already know is discrete.  However, a much different argument is
needed to prove that the tiles cover $\E$ in the affine case.

\begin{figure}[htb]
\centerline{\epsfxsize=3in \epsfbox{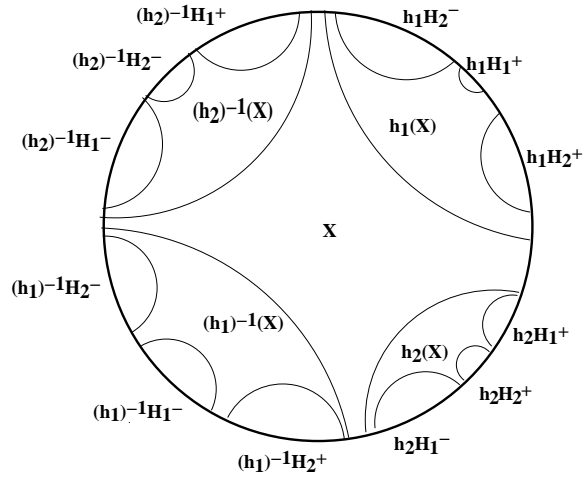}}
\caption{The tiling associated to a Schottky group, in the Poincar\'e disk model}
\label{fig:schottky1}
\end{figure}

\begin{lem}\label{lem:ldisjoint}
If $g\in G$ is nontrivial, then $g\Delta\cap\Delta = \emptyset$.
\end{lem}

\begin{lem}\label{lem:ldiscrete} G is discrete.\end{lem}

\begin{lem}\label{lem:lcomplete} The union
\begin{equation*}
G\bD = \bigcup_{g\in G} g(\bD)
\end{equation*}
equals all of $\rht$.
\end{lem}
\begin{proof}[Proof of Lemma~\ref{lem:ldisjoint}]
We show that if $g\in G$ is a reduced word
\begin{equation*}
g = g_{i_k}^{j_k} \dots g_{i_1}^{j_1} 
\end{equation*}
(where $j_i = \pm 1$) then either $k=0$ (that is, $g = 1$) or 
$g \Delta\cap \Delta = \emptyset$. 
In the latter case $g\bD\cap\bD=\emptyset$ 
unless $k=1$. This implies that the $g_i$ freely generate $G$
and that $G$ acts properly and freely on the union
\begin{equation*}
\Gamma\bD = \bigcup_{\gamma\in\Gamma} \gamma\bD.
\end{equation*}
Then we show that $\Gamma\bD = \rht$.

Suppose that $k>0$. We claim inductively that 
$g(\Delta) \subset H_{i_k}^{j_k}.$
Let 
$g' = g_{i_{k-1}}^{j_{k-1}} \dots g_{i_1}^{j_1} $ 
so that $g = g_{i_k}^{j_k} g'$.
If $k=1$, then $g'=1$ and 
\begin{equation*}
g \Delta =
g_{i_1}^{j_1} \Delta \subset 
g_{i_1}^{j_1} (\rht - \Bar{H}_{i_1}^{-j_1})\subset
H_{i_1}^{j_1}\subset \rht -\Delta.
\end{equation*}
If $k>1$, then $g_{i_k}^{-j_k}\neq
g_{i_{k-1}}^{j_{k-1}}$ (since $g$ is a
reduced word) so 
$H_{i_k}^{-j_k}\neq H_{i_{k-1}}^{j_{k-1}}$. The
induction hypothesis
\begin{equation*}
g'(\Delta)\subset H_{i_{k-1}}^{j_{k-1}}
\end{equation*}
implies
\begin{equation*}
g \Delta =
g_{i_k}^{j_k}g' \Delta \subset
g_{i_k}^{j_k} H_{i_{k-1}}^{j_{k-1}} \subset
g_{i_k}^{j_k} (\rht - \Bar{H}_{i_k}^{-j_k})  \subset 
H_{i_k}^{j_k}
\end{equation*}
as desired. Thus all of the $\gamma\Delta$, for $\gamma\in\Gamma$,
are disjoint and their closures $\gamma\bD$ tile $\Gamma\bD$.
\end{proof}
\begin{proof}[Proof of Lemma~\ref{lem:ldiscrete}]
Let $x_0\in\Delta$. Since $\Delta$ is open, there exists $\delta_0>0$
such that the $\delta_0$-ball (in the hyperbolic metric $d$ on $\rht$)
about $x_0$ lies in $\Delta$. We have proved that if $g$ is a reduced
word in $g_1,\dots,g_m$ then $g(x_0)\notin \Delta$ so
$d(x_0,g(x_0))>\delta_0$. In particular no sequence of reduced words in $G$
can accumulate on the identity, proving that $G$ is discrete.
\end{proof}

% The first part of the proof shows that, in addition to
% \begin{equation*}
% g(\Delta)\subset H_{i_k}^{j_k}
% \end{equation*}
% whenever $g = g_{i_k}^{j_k} \dots g_{i_1}^{j_1}$ is a reduced word,
% \begin{equation}\label{eq:nestedhspaces}
% g(H_i^j)\subset H_{i_k}^{j_k} 
% \end{equation}
% whenever $(i,j)\neq(i_k,-j_k)$.

\begin{proof}[Proof of Lemma~\ref{lem:lcomplete}]
We must prove $G\bD=\rht$. We use a completeness argument to show that
the quotient $M = (G\bD)/G$ is actually $\rht/G$. 
We begin by making an abstract model for the universal covering space
$\tilde M$ as the quotient space of the Cartesian
product $\bD\times G$ by the equivalence relation generated by identifications
\begin{equation*}
(x,g) \sim  (g_i^j x, g g^{-j}_i) 
\end{equation*}
where $x\in\partial H^{-j}_i$, $g\in G$ and $(i,j)\in I\times J$. 
Then $M = (G\bD)/G$ inherits a hyperbolic structure from $\rht$.
We show that this hyperbolic structure is complete to prove that
$M=\rht/G$ and thus $B\bD=\rht$.

The identifications define the structure of a smooth manifold on $\tilde M$.
If $x\in\Delta$, then the equivalence class of $(x,g)$ contains only
$(x,g)$ itself, and a smooth chart at $(x,g)$ arises from the smooth
structure on $\Delta$. If $x\in\partial\Delta$, then the equivalence class
of $(x,g)$ equals
\begin{equation*}
\{  (x,g), (g_i^j x, g g^{-j}_i)\} 
\end{equation*}
where $x\in\partial H^{-j}_i$. 
Let $U$ be an open neighborhood of $x$ in $\rht$ which intersects the orbit
$\Gamma x$ in $\{x\}$.
Since $x$ is a boundary point of the smooth
surface-with-boundary $\bD$, the intersection $U\cap\bD$ is a coordinate
patch for $x$ in $\bD$. The image $g_i^j(U-\Delta)$ is a smooth coordinate
patch for $g_i^j x$ in $\bD$ and the image of the union
\begin{equation*}
(U\cap\bD)\times \{g\}\quad \bigcup \quad 
(g_i^j(U-\Delta)) \times \{gg_i^{-j}\} 
\end{equation*}
is a smooth coordinate patch for the equivalence class of $(x,g)$.

The $G$-action defined by
\begin{equation*}
\gamma: (x,g) \longmapsto (x,\gamma g)
\end{equation*}
preserves this equivalence relation and thus defines a $G$-action 
on the quotient $\tilde M$. The map
\begin{align*}
D: \bD\times G   & \longrightarrow  \rht \\
(x,g) & \longmapsto g(x)
\end{align*}
preserves the equivalence relation and defines a $G$-equivariant
map, the {\em developing map\/} $D:\tilde M\longrightarrow\rht$.
The developing map $D$ is a local diffeomorphism onto the open set
$G\bD$. 

Pull back the hyperbolic metric from $\rht$  by $D$ to obtain a Riemannian
metric on $\tilde M$ for which $D$ is a local isometry. Since $G$ acts
isometrically on $\rht$, the developing map $D$ is $G$-equivariant.
We claim $\tilde M$ is geodesically complete. To this end, consider a
maximal unit-speed geodesic ray $\mu(t)$ defined for $0 < t < t_0$.
Its preimage in $\bD\times G$ consists of geodesic segments $\mu_k$ in 
various components $\bD\times \{g_{i_k}\}$. 

We claim there are only finitely many segments $\mu_k$.  Since $\bD$
is convex, one of its segments $\mu_k$ cannot enter and exit $\bD$
from the same side.  Since the defining arcs $A_i^j$ are pairwise
disjoint, the corresponding geodesics $\partial H_i^j$ are pairwise
ultraparallel and the distance between different sides of $\bD$ is
bounded below by $\delta >0$. Since the length of $\mu$ equals $t_0$,
there can be at most $t_0/\delta$ segments $\mu_k$.  

Let $\mu_k:[t_1,t_0)\longrightarrow \bD\times \{g_{i_k}\}$ be the last
geodesic segment.  Since $\bD$ is closed, $\mu_k(t)$ converges as
$t\longrightarrow t_0$, contradicting maximality.

Thus $\tilde M$ is geodesically complete. Since a local isometry
from a complete Riemannian manifold is a covering space,
$D$ is a covering space. The van Kampen 
theorems imply that $\tilde M$ is simply connected, so that $D$ is
a diffeomorphism and hence surjective. Thus $G\bD=\rht$ as desired.
\end{proof}

\subsection{Existence of a small interval}

The following lemma is an elementary fact which is used in the proof
of completeness. We assume that the number $m$ of generators in the 
Schottky group is at least $2$.

\begin{lem} 	\label{lem:small} 
Let $\{A_i^j\mid (i,j)\in I\times J\}$ 
be a collection of disjoint intervals on $S^2\cap\Nn_+$. %let's call this S^1
% Suppose that $m>1$. 
Then there exists an $(i_0,j_0)\in I\times J$ such
that the length $\Phi(A_{i_0}^{j_0}) < \pi/2$.
\end{lem}

\begin{proof}
Since the $A_i^j$ are disjoint, their total length is bounded by $2\pi$.
Since there are $2m \ge 4$ of them, 
at least one has length $< (2\pi)/4$ as claimed.
\end{proof}

\subsection{A criterion for $\epsilon$-hyperbolicity}\label{sec:hypcrit}
To determine proper discontinuity of affine Schottky groups, we 
examine sequences of group elements.  
An important case is when for some $\epsilon >0$, every
element of the sequence is $\epsilon$-hyperbolic.  Here is a useful criterion
for such $\epsilon$-hyperbolicity of an entire sequence.  As before,
$A_i^j, (i,j)\in I\times J$,
%correction  (inserted comma)
denote the disjoint intervals associated to the
generators $g_1,\dots ,g_m$ of a Schottky group.
\begin{lem}\label{lem:fxpts}
Let $\theta_0$ be the minimum angular separation between the 
intervals $A_i^j\subset S^1$ and let 
$\epsilon_0 = 2\sin(\theta_0/2)$.
Suppose that 
\begin{equation}\label{eq:reducedword}
g = g_{i_0}^{j_0}\dots g_{i_l}^{j_l}
\end{equation}
is a reduced word. 
If $(i_l,j_l)\neq (i_0,-j_0)$ then $g$ is $\epsilon_0$-hyperbolic.
\end{lem}
The condition $(i_l,j_l)\neq (i_0,-j_0)$ means that 
\eqref{eq:reducedword} describes a {\em cyclically reduced\/} word.
\begin{proof}
By Lemma~\ref{lem:Brouwer},
\begin{equation*}
\xp{g}\in A_{i_0}^{j_0} 
\end{equation*}
and 
\begin{equation*}
\xm{g} = \xp{g^{-1}} \in A_{i_l}^{-j_l}.
\end{equation*}
Since $(i_l,j_l)\neq (i_0,-j_0)$, the vectors
$\xm{g}$ and $\xp{g}$ lie in the attracting interval and repelling
interval respectively. Since they lie in  disjoint conical neighborhoods,
they are separated by at least $\theta_0$.
\end{proof}
This lemma %will be used crucially to analyze 
is crucial in the analysis of %correction
sequences of elements of
$\Gamma$ arising from incompleteness. These sequences will all
be $\epsilon$-hyperbolic for some $\epsilon>0$. 
A typical sequence which is not $\epsilon$-hyperbolic for any $\epsilon >0 $
is the following (compare \cite{DrummGoldman1}): 
\begin{equation*}
\gamma_n = g_1^n g_2 g_1^{-n}.
\end{equation*}
Since all the elements are conjugate, the eigenvalues are constant 
(in particular they are bounded). 
However both sequences of eigenvectors $\xm{\gamma_n}, \xp{\gamma_n}$ 
converge to $\xp{g_1}$ so 
$\rho(\xm{\gamma_n}, \xp{\gamma_n}) \longrightarrow 0.$
 
\section{Crooked planes and zigzags}\label{sec:crooked}
Now consider the action of a Schottky group $G\subset\SOoo$ on Minkowski 
$(2+1)$-space $\E$. 
The inverse projectivization $\P^{-1}(\Delta)$ of a fundamental
domain $\Delta$ is a fundamental domain for the linear action of $G$
on the subspace $\Nn$ of timelike vectors.
We extend this fundamental domain to a 
larger open subset of $\E$. The extended fundamental domains are
bounded by polyhedral surfaces called {\em crooked planes.\/} For the
groups of linear transformations, the crooked planes all pass through
the origin --- indeed the origin is a special point of each crooked
plane, its {\em vertex.\/}

The Schottky group $G = \la g_1,\dots,g_m\ra$ acts properly
discontinuously and freely on the open subset $\Nn$ consisting of
timelike vectors in $\rto$.  However, since $G$ fixes the origin, the
$G$-action on all of $\E$ is quite far from being properly discontinuous.  

Then we deform $G$ inside the group of affine isometries of $\E$ to
obtain a group $\Gamma$ which in certain cases acts freely and
properly discontinuously. This {\em affine deformation\/} $\Gamma$ of
$G$ is defined by geometric identifications of a family of {\em disjoint\/} 
crooked planes. The crooked planes bound {\em crooked half-spaces\/}
whose intersection is a {\em crooked polyhedron\/} $X$.  
In \cite{Drumm1}, Drumm proved the remarkable fact that as long as the
crooked planes are disjoint, 
$\Gamma$ acts freely and properly discontinuously on $\E$ with fundamental 
domain $X$.

\subsection{Extending Schottky groups to Minkowski space}
\label{sec:extSchottky1}

When $\Delta$ is a fundamental domain for $G$
acting on $\rht$, its inverse image $\P^{-1}(\Delta)$ is a fundamental domain for the
action of $G$ on $\Nn$.  The faces of $\P^{-1}(\Delta)$ are the
intersections of $\Nn$ with indefinite planes corresponding to the
geodesics in $\rht$ forming the sides of $\Delta$.

Each face $S$ of $\P^{-1}(\Delta)$ extends to a polyhedral surface $\Cc$
in $\rto$ called a {\em crooked plane.\/} The face
$S\subset\P^{-1}(\Delta)$ is the {\em stem\/} of the crooked plane. 
To extend the stem one adds two null half-planes, called the {\em wings,\/} 
along the null lines bounding $S$. 
Crooked planes are more flexible than Euclidean planes
since one can build fundamental polyhedra for free, properly discontinuous
groups from them.

Crooked planes were introduced by Drumm in his doctoral 
dissertation~\cite{Drumm0} (see also \cite{Drumm1,Drumm2}). 
Their geometry is extensively discussed in Drumm-Goldman~\cite{DrummGoldman2}
(see also \cite{ERA}) where their intersections are classified.

\subsection{Construction of a crooked plane}\label{defcp}
Here is an example from which we derive the general definition of a crooked 
plane. The geodesic $l_\vv$ determined by the spacelike vector
\begin{equation*}
\vv = \bmatrix 1 \\ 0 \\ 0 \endbmatrix 
\end{equation*}
corresponds to the set
\begin{equation*}
S_0 = \overline{(\P^{-1}(l_\vv))} 
= \left\{ \bmatrix 0 \\ \vu_2 \\ \vu_3 \endbmatrix \Bigg| 
\vert \vu_3\vert \geq \vert \vu_2\vert   \right\} \subset \Bar{\Nn}
\end{equation*}
which is the {\em stem} of the crooked plane. The two lines bounding $S_0$ are
\begin{align*}
\partial^-S_0 & =  
\left\{ \bmatrix 0 \\ \vu_2 \\ \vu_3 \endbmatrix \Bigg| \vu_2 = \vu_3 \right\} 
\\
\partial^+S_0 & =  
\left\{ \bmatrix 0 \\ \vu_2 \\ \vu_3 \endbmatrix \Bigg| 
\vu_2 = -\vu_3 \right\} 
\end{align*}
and the {\em wings} are the half-planes
\begin{align*}
\Wingm_0 & =  
\left\{ \bmatrix \vu_1 \\ \vu_2 \\ \vu_3 \endbmatrix 
\Bigg| 
\vu_1 \le 0, \vu_2 = \vu_3 \right\} \\
\Wingp_0 & =  
\left\{ \bmatrix \vu_1 \\ \vu_2 \\ \vu_3 \endbmatrix \Bigg| 
\vu_1 \ge 0, \vu_2 = -\vu_3 \right\}. 
\end{align*}
The crooked plane 
$\Cc_0$ is defined as the %correction (inserted ``defined as'' and $\Cc-0$)
the union 
\begin{equation*}
\Cc_0 = \Wingm_0 \cup S_0 \cup \Wingp_0.
\end{equation*}
(Compare Figure~\ref{fig:cp1}.)
Corresponding to the half-plane $H_\vv\subset\rht$ is the region
\begin{equation*}
\tilde{H}_\vv = 
\left\{ \bmatrix \vu_1 \\ \vu_2 \\ \vu_3 \endbmatrix \in\Nn \ \Bigg| 
\vu_1 > 0, \vu_3 > 0 \right\}
\end{equation*}
and the component of the complement $\E - \Cc_0$ containing
$\tilde{H}_\vv$ is the {\em crooked half-space\/}
\begin{align*}
\Hh(\vv) & = 
\left\{ \bmatrix \vu_1 \\ \vu_2 \\ \vu_3 \endbmatrix \ \Bigg| 
\vu_1 > 0, \vu_2 + \vu_3 > 0 \right\} \\ 
& \quad \bigcup
\left\{ \bmatrix \vu_1 \\ \vu_2 \\ \vu_3 \endbmatrix \ \Bigg| 
\vu_1 = 0, \vu_2 + \vu_3 > -0, \vu_2 + \vu_3 > 0 \right\} \\
& \qquad \bigcup
\left\{ \bmatrix \vu_1 \\ \vu_2 \\ \vu_3 \endbmatrix \ \Bigg| 
\vu_1 < 0, -\vu_2 + \vu_3 > 0 \right\} .
\end{align*}
%correction   put a period after the last display
% (Compare Figure~\ref{fig:sideview}.)
Now let $\vu\in\Ss$ be any unit-spacelike vector, determining the half-space
$H_\vu\in\rht$. The crooked plane directed by $\vu$ can be defined as
follows, using the previous example.  Let $g\in\SOoo$ such that
$g(\vv)=\vu$.  The {\em crooked plane directed by $\vu$\/} is 
$\Cc(\vu) =  g(\Cc_0)$.

Since $g$ preserves the spacelike, lightlike or timelike nature of a
vector, we see that $\Cc(\vu)$ is composed of a stem flanked by two
tangent wings, just like $\Cc_0$.

The crooked plane is singular at the origin,
which we call the {\em vertex\/} of $\Cc(\vu)$. In general, if $p\in\E$
is an arbitrary point, the {\em crooked plane directed by $\vu$ and
with vertex $p$\/} is defined as:
\begin{equation*}
\Cc(\vu,p) =  p +\Cc(\vu). 
\end{equation*}
Let $\Cc(\vu,p)\subset\E$ be a crooked plane. 

We define the crooked half-space $\Hh(\vu,p)$ to be the component
of the complement $\E-\Cc(\vu,p)$ which is bounded by $\Cc(\vu,p)$
and contains $p+H_{\vu}$. Note that $\E$ decomposes as a disjoint union
\begin{equation*}
\E = \Hh(\vu,p) \cup \Cc(\vu,p) \cup \Hh(-\vu,p).
\end{equation*}
(Crooked half-spaces are called  {\em wedges\/} in 
Drumm~\cite{Drumm0,Drumm1,Drumm2}.)

The {\em angle} $\Phi(\Hh(\vu,p))$ of the crooked half-space
$\Hh(\vu,p)$ is taken to be the angle of $A$, the interval determined
by the half-space $H_\vu$.

\begin{figure}[ht]
\centerline{\epsfxsize=3in \epsfbox{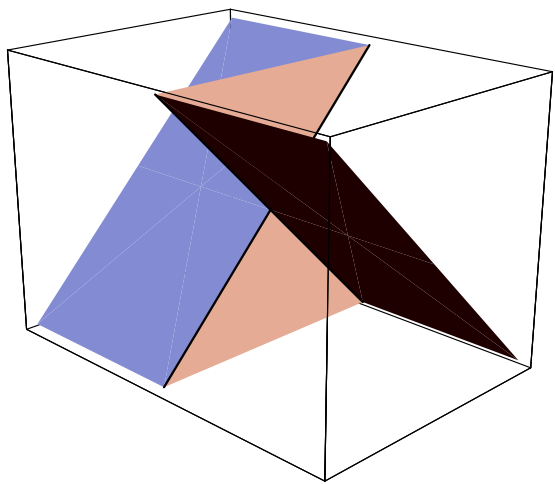}}
\caption{A crooked plane}\label{fig:cp1}
\end{figure}

% \begin{figure}[ht]
% \centerline{\epsfxsize=3in \epsfbox{sideview.eps}}
% \caption{Side view of a crooked half-space}\label{fig:sideview}
% \end{figure}

% \begin{lem}\label{lem:wuhalf}
% Let $q\in\E$ and $\vv\in\nn$.
% Let $\HH$ be a crooked half-space with vertex $q$. 
% Consider the null line $L = q + \R\vv$  and 
% the corresponding null plane $P = q + \R\vv^\perp$.
% If $L$ lies in the conical neighborhood corresponding to the stem
% of $\HH$, then $\HH$ intersects every plane parallel to $\vv^\perp$
% in a half-plane.
% \end{lem}
% We may apply this to the weak-unstable plane of an hyperbolic
% isometry.  By Lemma~\ref{lem:Brouwer}, the vector $\xp{g_i^j}$ lies in
% the attracting interval $A_i^j$. The null plane corresponding to
% $\xp{g_i^j}$ is the weak-unstable plane $E^{wu}(g_i^j)$ (see
% Definition~\ref{def:weakunstable}).  Since the stem of the crooked
% plane $\Cc_i^j$ cooresponds to $A_i^j$, Lemma~{lem:wuhalf} implies
% that the corresponding null plane $E^wu(g_i)$ intersects $\Hh_i^j$ in
% a half-plane.

% \begin{figure}[ht]
% \centerline{\epsfxsize=3in \epsfbox{hyp1.eps}}
% \caption
% {Weak-unstable ray in a zigzag region}
% \label{fig:hyp1}
% \end{figure}

\subsection{Zigzags}\label{sec:zigzags}
To understand the tiling of $\E$ by crooked polyhedra, we intersect
the tiling with a fixed definite plane $P$, which is always transverse
to the stem and wings of any crooked plane.  Since the tiling only
contains countably many crooked planes, we may assume that $P$ misses
the vertices of the crooked planes in the tiling.
 
A {\em zigzag\/} in a definite plane $P$ is a union $\zeta$ of two disjoint
rays $r_0$ and $r_1$ and the segment $s$ (called the {\em stem\/})
joining the endpoint $v_0$ of $r_0$ to the endpoint $v_1$ of $r_1$, 
such that the two angles $\theta_0$ and $\theta_1$  formed by the rays at the
respective endpoints of $s$ differ by $\pi$ radians. 
The intersection of a crooked plane with a definite plane not
containing its vertex is a zigzag.  (Conversely, every zigzag extends to
a unique crooked plane, although we do not need this fact.)  Compare Figure~\ref{fig:crooked}.

\begin{figure}[ht]
\centerline{\epsfxsize=3in \epsfbox{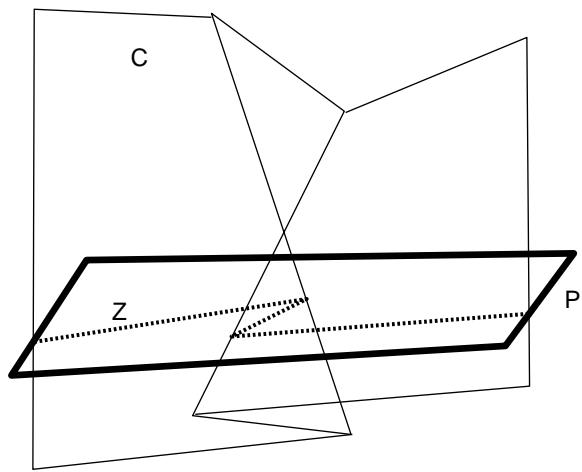}}
\caption
{The intersection of a crooked plane $C$ with a definite plane $P$ 
is a zigzag $Z$.}
\label{fig:crooked}
\end{figure}

\begin{figure}[htb]
\centerline{\epsfxsize=3in \epsfbox{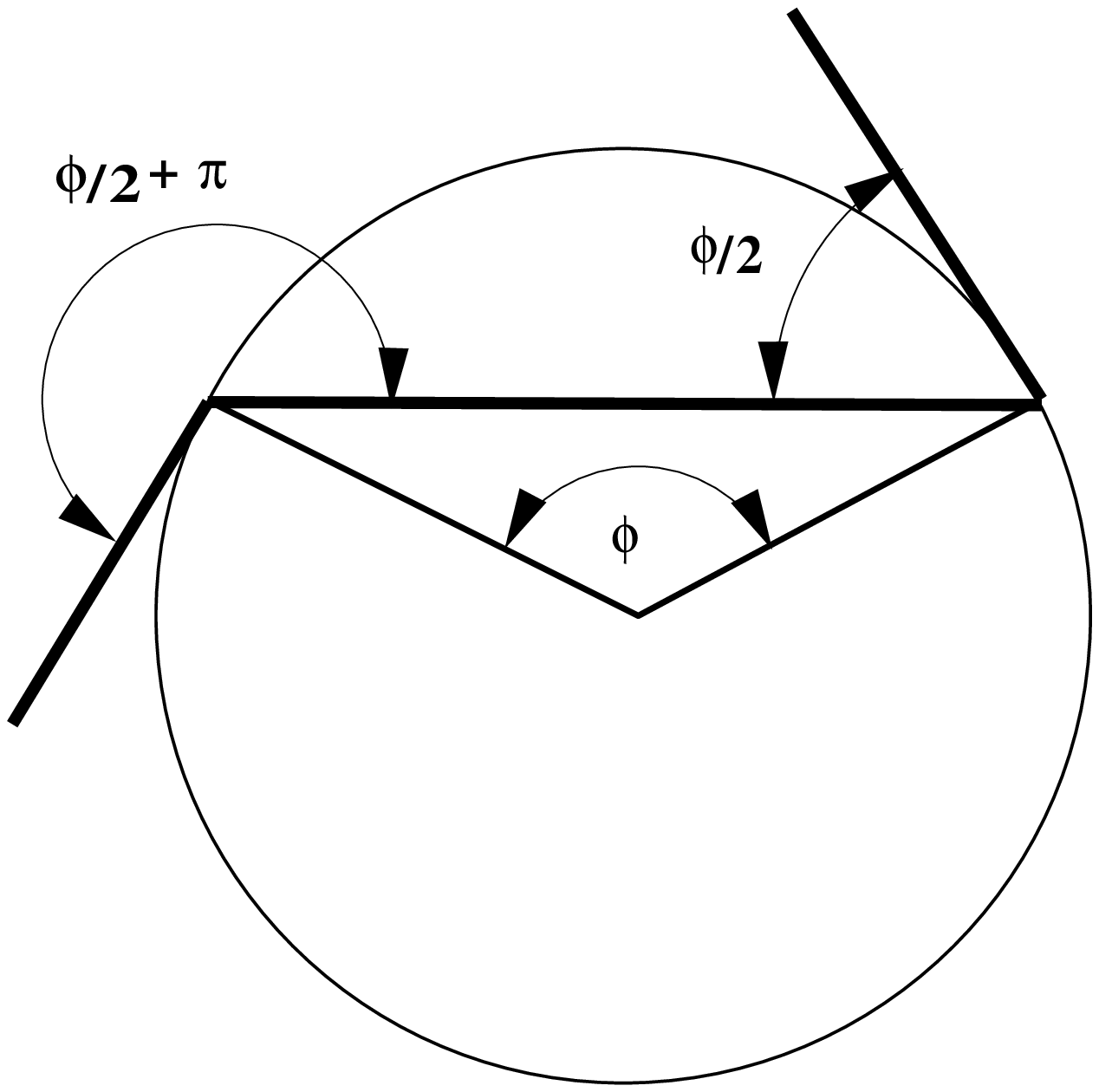}}
\caption{Angles in a zigzag}
\label{fig:angles}
\end{figure}

% \begin{thm}
% If $\Cc\subset \E$ is a crooked plane whose vertex does not lie on $P$, 
% then $\Cc\cap P$ is a zigzag in $P$. If $\zeta\subset P$ is a zigzag, a unique
% crooked plane $\Cc$ intersects $P$ in $\zeta$.
% \end{thm}
% \begin{proof}
% Let $\zeta$ be a zigzag. The stem of $\zeta$ is a chord of the unique
% circle $\sigma$ tangent to the rays of $\zeta$. Let $p$ be the vertex of the
% null cone $N(p)$ such that $\sigma = N(p)\cap P$. The union of all lines
% containing $p$ is the stem of a unique crooked plane $\Cc$ (with vertex at
% $p$) such that $\Cc \cap P = \zeta$.
% \end{proof}

A {\em zigzag region\/} is a component $\zz$ of $P-\zeta$ where
$\zeta\subset P$ is a zigzag. Equivalently $\zz$ is the
intersection of a crooked half-space with $P$.
The corresponding half-space in $\rht$ is bounded by an interval $A\subset S^1$
whose length $\Phi(A)$ is defined as the 
{\em angle\/} $\Phi(\zz)$ of the zigzag region. One of the angles of
$\zz$ 
equals $\Phi(A)/2$ and the other, $\Phi(A)/2+\pi$.
Compare Figure~\ref{fig:angles}.

%%% maybe this next paragraph can be polished too
If $r\subset P$ is an open ray contained in the zigzag region $\zz$,
then $r$ lies in a unique maximal open ray inside $\zz$ (just move the
endpoint of $r$ back until it meets $\zeta$). Every maximal open ray
then lies in one of two angular sectors (possibly both).  One angular
sector has vertex $v_0$ and sides $r_0$ and $s$, and subtends the
angle $\theta_0$.  The other angular sector has vertex $v_1$ and sides
$r_1$ and $s$, and subtends the angle $\theta_1$.  Any two rays in $R$
make an angle of at most $\Phi(\zeta)$ with each other.

\subsection{Affine deformations}\label{sec:affineSchottky}
Consider a group $\Gamma$ generated by isometries
\begin{equation*}
h_1,\dots, h_m\in\iso  
\end{equation*}
for which there exist crooked half-spaces $\Hh_i^j$, 
where $(i,j)\in I\times J$
(using the indexing convention defined in \S\ref{sec:Schottky})
such that:
\begin{itemize}
\item all the $\Hh_i^j$ are pairwise disjoint;
\item $h_i(\Hh_i^-)= \E - \Bar{\Hh}_i^+$. (Thus
$h_i^{-1}(\Hh_i^+)= \E - \Bar{\Hh}_i^-$ as well.) 
\end{itemize}
We call $\Gamma$ an {\em affine Schottky group.\/}  The set 
\begin{equation*}
X = \E - \bigcup_{(i,j)\in I\times J} \Bar{\Hh}_i^j 
\end{equation*}
is an open subset of $\E$ whose closure $\bX$ is a finite-sided polyhedron
in $\E$ bounded by the crooked planes $\Cc_i^j=\partial\Hh_i^j$.

Drumm-Goldman~\cite{DrummGoldman2} provides criteria for
disjointness of crooked half-spaces. 
In particular for every configuration of disjoint half-planes
\begin{equation*}
H_1^+,H_1^-,\dots,H_m^+,H_m^-\subset\rht 
\end{equation*}
paired by Schottky generators 
$g_1,\dots,g_m\in\SOoo$, there exists a configuration
of disjoint crooked half-spaces
\begin{equation*}
\Hh_1^+,\Hh_1^-,\dots,\Hh_m^+,\Hh_m^-\subset\E
\end{equation*}
whose directions correspond to the $H_i^j$, and which are paired by affine transformations
$h_i$ with linear part $g_i$, satisfying the above conditions.

We show that $\bX$ is a fundamental polyhedron for $\Gamma$ acting 
on $\E$. As with standard Schottky groups, one first shows that
the images $h\bX$ form a set of {\em disjoint\/} tiles of $\Gamma\bX$.

\begin{lem} $\Gamma\bX$ is open.\end{lem}
\begin{proof}
(Compare the proof of Lemma~\ref{lem:lcomplete}.)
%correction
If $x\in X$, then $\gamma x$ is an interior point of $\Gamma X \subset
\Gamma\Bar{X}$ for every $\gamma\in\Gamma$.
Otherwise suppose $x\in\partial X$. Then 
$x\in \Cc_i^j$ for some $(i,j)\in I\times J$.
Let $B$ be an open ball about $x$ such
that $B\cap\partial X\subset \Cc_i^j$. Then 
\begin{equation*}
(B\cap\bX)  \cup (h_i^{-j}B\cap\bX) \subset \bX
\end{equation*}
is an open subset of $\bX$ whose orbit is an open neighborhood of $x$
in $\Gamma\bX$.
\end{proof}
The analogue of Lemma~\ref{lem:ldisjoint} is:
\begin{lem}\label{lem:adisjoint}
The affine transformations $h_1,\dots,h_m$ freely generate $\Gamma$. The crooked polyhedron
$\bX$ is a fundamental domain for $\Gamma$ acting on $\Gamma\bX$.
\end{lem}
\begin{proof}
The proof is completely identical to that of
Lemma~\ref{lem:ldisjoint}.  Replace the hyperbolic half-spaces $H_i^j$
by crooked half-spaces $\Hh_i^j$, the hyperbolic polygon $\Delta$
by the crooked polyhedron $X$ and the hyperbolic isometries $g_i$ by
Lorentzian affine isometries $h_i$.
\end{proof}
% Let $\gamma\in\Gamma$ be a reduced word 
% \begin{equation*}
% \gamma = h_{i_1}^{j_1} \dots h_{i_k}^{j_k} 
% \end{equation*}
% (where $j_l = \pm 1$ and $k>0$).  We show that 
% $\gamma X\cap X = \emptyset$. Furthermore $\gamma\bX\cap\bX=\emptyset$
% unless $k=1$.
% 
% We claim inductively that 
% \begin{equation*}
% \gamma(X)\subset \Hh_{i_1}^{j_1}. 
% \end{equation*}
% Let 
% \begin{equation}\label{eq:gammaprime}
% \gamma' = h_{i_2}^{j_2} \dots h_{i_k}^{j_k} 
% \end{equation}
% so that
% \begin{equation*}
% \gamma = h_{i_1}^{j_1} \gamma'
% \end{equation*}
% where $\gamma'$ is a reduced word.
% If $k=1$, then $\gamma'=\Id$ and 
% \begin{equation*}
% \gamma X =
% h_{i_1}^{j_1} X \subset
% h_{i_1}^{j_1} (\E - \Bar{\Hh}_{i_1}^{-j_1})\subset
% \Hh_{i_1}^{j_1}\subset \E-X.
% \end{equation*}
% If $k>1$, then $h_{i_1}^{-j_1}\neq
% h_{i_2}^{j_2}$ (since $\gamma$ is a reduced word) so 
% $\Hh_{i_1}^{-j_1}\neq \Hh_{i_2}^{j_2}$. The
% induction hypothesis
% \begin{equation*}
% \gamma'(X)\subset \Hh_{i_2}^{j_2}
% \end{equation*}
% implies
% \begin{equation*}
% \gamma X =
% h_{i_1}^{j_1}\gamma' X \subset
% h_{i_1}^{j_1} \Hh_{i_2}^{j_2} \subset
% h_{i_1}^{j_1} (\E - \Bar{\Hh}_{i_1}^{-j_1})  =
% \Hh_{i_1}^{j_1}
% \end{equation*}
% as desired.
% \end{proof}
The most difficult part of the proof of the Main Theorem is that 
$\Gamma\bX=\E$, 
that is, the images $\gamma\bX$ tile {\em all\/} of $\E$. 
Due to the absence of an invariant {\em Riemannian\/} metric,
the completeness proof of Lemma~\ref{lem:lcomplete} fails.

\section{Completeness}
% Now we 
We %correction 
prove that the images of the crooked polyhedron $X$ tile
$\E$. We suppose there exists a point $p$ not in 
$\Gamma\bX$ and derive a contradiction.

The first step is to describe a sequence of nested crooked half-spaces
$\HH_k$ containing $p$.  This sequence corresponds to a sequence of
indices 
\begin{equation*}
(i_0,j_0), (i_1,j_1),\dots,(i_k,j_k),\dots 
\end{equation*}
%correction made that sequence a display so that it doesn't overflow the line
such that $\HH_k
= \gamma_k \Hh_{i_k}^{j_k}$ where $\gamma_k = h_{i_0}^{j_0}\dots
h_{i_{k-1}}^{j_{k-1}}$.

Since crooked polyhedra are somewhat complicated and the elements of
$\Gamma$ exhibit different dynamical behavior in different directions,
bounding the separation of the crooked polyhedra requires some care.
To simplify the discussion we intersect this sequence with a fixed
definite plane $P$ so that the crooked half-spaces $\HH_k$ intersect
$P$ in a sequence of nested zigzag regions containing $p$.
We then approximate the zigzag regions by half-planes $\Pi_k\subset P$
(compare Figures~\ref{fig:cluster} and \ref{fig:seqzz})
%correction			reference to pictures
and show that for infinitely many $k$, the distance between the successive
lines $L_k=\partial\Pi_k$ bounding $\Pi_k$ is bounded below, to reach
the contradiction. The Compression Lemma~\ref{lem:CompressionLemma}
gives a lower bound for $\rho(L_k,L_{k+1})$ whenever $\gamma_k$ is 
$\epsilon$-hyperbolic. Using the special form of the sequence
$\gamma_k$ and 
the Hyperbolicity Criterion (Lemma~\ref{lem:fxpts}), we find infinitely
many $\epsilon$-hyperbolic $\gamma_k$ for some $\epsilon>0$ and achieve
a contradiction.

\subsection{Construction of the nested sequence}\label{sec:sequence}
The complement $\E - \bX$ consists of the $2m$ crooked half-spaces
$\Hh_i^j$, which are  
bounded by crooked planes $\Cc_i^j$, 
indexed by $I\times J$.

\begin{lem}\label{lem:nestedsequence}
Let $p\in\E - \Gamma\bX$. There exists a sequence
$\{\HH_k\}$  of crooked half-spaces such that
\begin{itemize}
\item $\HH_k\supset\HH_{k+1}$ and $\HH_k\neq\HH_{k+1}$;
\item $p\in\HH_k$;
\item there exists a sequence $(i_0,j_0),(i_1,j_1),\dots,(i_n,j_n),\dots$
in $I\times J$ such that 
$(i_k,j_k)\neq (i_{k+1},-j_{k+1})$ for all $k\ge 0$ and
\begin{equation*}
\HH_k = \gamma_k \Hh_{i_k}^{j_k}
\end{equation*}
where
\begin{equation*}
\gamma_k = h_{i_0}^{j_0} h_{i_1}^{j_1} \dots h_{i_{k-1}}^{j_{k-1}}. 
\end{equation*}
\end{itemize}
\end{lem}
\begin{proof}
We first adjust $p$ so that the first crooked half-space $\HH_0$ satisfies
$\Phi(\HH_0)<\pi/2$.
By Lemma~\ref{lem:small}, 
\begin{equation}\label{eq:firstonesmall}
\Phi(A_{i_0}^{j_0})<\pi/2  
\end{equation}
for some $(i_0,j_0)\in I\times J$.
Since $p\notin\bX$, there exists $(i,j)$ such that $p\in\Hh_i^j$.  If
$(i,j)\neq(i_0,j_0)$, then we replace $p$ by $\gamma p$, for some
$\gamma\in\Gamma$ such that $\gamma p\in\Hh_{i_0}^{j_0}$.  Here is how
we do this.  If
$(i,j)\neq(i_0,-j_0)$, then $\gamma=h_{i_0}^{j_0}$ moves $p$ into
$Hh_{i_0}^{j_0}$. 
%%% fix this
Otherwise first move $p$ into
a crooked half-space other than $\Hh_{i_0}^{-j_0}$, then into 
$\Hh_{i_0}^{j_0}$. For example, $h_{i_1}^{j_1}$ moves $p$ into
$\Hh_{i_1}^{j_1}$ and then $\gamma = h_{i_0}^{j_0}h_{i_1}^{j_1}$ moves $p$ 
into $\Hh_{i_0}^{j_0}$. Thus we may assume that
\begin{equation*}
p \in \HH_0 =  \Hh_{i_0}^{j_0}
\end{equation*}
where $\Phi(\HH_0)<\pi/2$.

Suppose inductively that 
\begin{equation*}
\HH_0 \supset \dots \supset \HH_k \ni p
\end{equation*}
is a nested sequence of crooked half-spaces containing $p$ satisfying
the conclusions of Lemma~\ref{lem:nestedsequence}. Then 
$\HH_k = \gamma_k \Hh_{i_k}^{j_k}$ and $\gamma_k^{-1}(p)\in 
\Hh_{i_k}^{j_k}$.  Let $\gamma_{k+1}=\gamma_k h_{i_k}^{j_k}$.  
Thus 
% \begin{equation*}
% \gamma_k^{-1}(p)\in\gamma_k^{-1}\HH_k = \Hh_{i_k}^{j_k}
% \end{equation*}
% and 
\begin{equation*}
\gamma_{k+1}^{-1}(p)\in h_{i_k}^{-j_k}\Hh_{i_k}^{j_k} = 
\E - \Hh_{i_k}^{-j_k}. 
\end{equation*}
Since $p\notin\Gamma\bX$,
\begin{equation*}
\gamma_{k+1}^{-1}(p)\in 
\E - \bX - \Hh_{i_k}^{-j_k} = 
\bigcup_{(i,j)\neq(i_k,-j_k)} \Hh_i^j. 
\end{equation*}
Let $(i_{k+1},j_{k+1})$ index the component
of $\E - \bX - \Hh_{i_k}^{-j_k}$ containing
$\gamma_{k+1}^{-1}(p)$. This gives the desired sequence.
\end{proof}

\subsection{Uniform Euclidean width of crooked polyhedra}\label{sec:uniform}
If $S\subset\Gamma\bX$, define the {\em star-neighborhood\/} of $S$ 
as the interior of the union of all tiles $\gamma\bX$ intersecting 
$\bar{S}$.

\begin{lem}\label{lem:delta0}
There exists $\delta_0>0$ such that 
the $\delta_0$-neighborhood $B(X,\delta_0)$ 
lies in the star-neighborhood of $\bX$.
In particular whenever $(i,j),(i',j')\in I\times J$ satisfy
$(i,j)\neq (i',-j')$,
\begin{equation}\label{eq:uniformdistance}
B(\E-\Hh_i^j,\delta_0) \subset  \E-h_i^j\bar{\Hh}_{i'}^{j'}.
\end{equation}
\end{lem}
\begin{proof}
The fundamental polyhedron $\bX$ is bounded by crooked planes
$\Cc_i^j=\partial\Hh_i^j$. 
The star-neighborhood of $\bX$ equals
\begin{equation*}
\bX \cup \bigcup_{(i,j)\in I\times J} h_i^j\bX.
\end{equation*}
Its complement consists of the $2m(2m-1)$ crooked half-spaces
$h_i^j \Hh_{i'}^{j'}$ where $(i',j')\neq(i,-j)$.
Unlike in hyperbolic space, two disjoint closed planar regions in Euclidean
space are separated a positive distance apart.
Four closed planar regions comprise a crooked plane.
Thus the distance between disjoint crooked
planes is strictly positive. Choose $\delta_0>0$ to be smaller than
the distance between any of the $\Cc_i^j$ and $h_i^j\Cc_{i'}^{j'}$.
The second assertion follows since the $\delta_0$-neighborhood of a crooked
half-space $\HH$ equals $\HH\cup B(\partial\HH,\delta_0)$.
\end{proof}
%correction 			new place for pictures
\begin{figure}[ht]
\centerline{\epsfxsize=3in \epsfbox{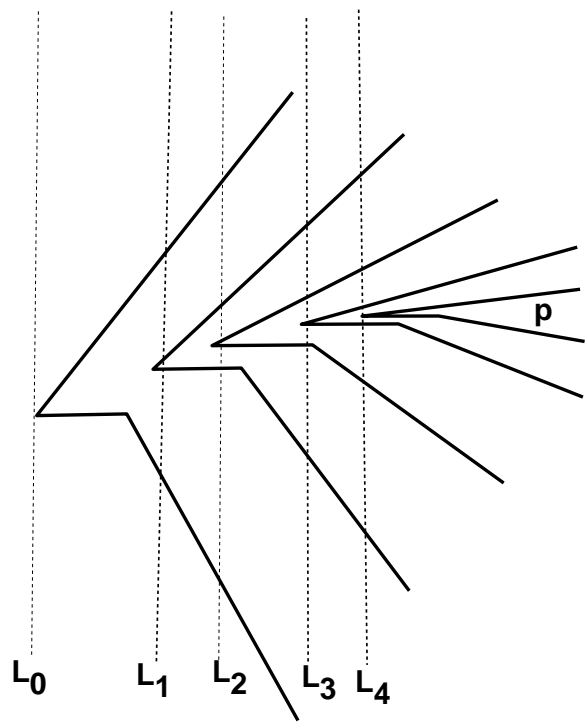}}
\caption
{Accumulating zigzag regions}
\label{fig:cluster}
\end{figure}

\begin{figure}[ht]
\centerline{\epsfxsize=3in \epsfbox{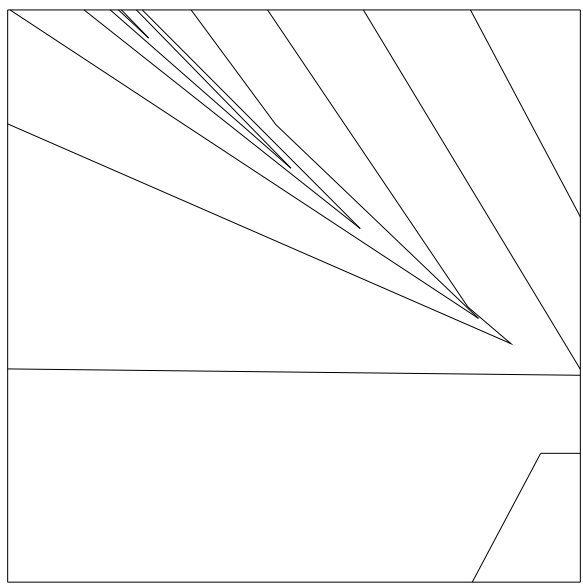}}
\caption
{A sequence of zigzag regions tending to infinity}
\label{fig:seqzz}
\end{figure}
%correction 			moved the pictures here
\subsection{Approximating zigzag regions by half-planes}\label{sec:approx}
Now intersect with $P$. 
We % will 
%correction
approximate each zigzag region $\HH_k\cap P$ by a Euclidean
half-plane $\Pi_k\subset P$ containing $p$. These half-planes form
a nested sequence
\begin{equation*}
\Pi_1 \supset \Pi_2 \supset \dots \supset \Pi_k \supset \dots
\end{equation*}
with $\HH_k\cap P\subset \Pi_k$ and we take the lines $L_k=\partial\Pi_k$
to be parallel. Since $p\in\HH_k\cap P$ for all $k$,
and $p\in\partial(\Gamma\bX)$, the parallel lines $L_k$ approach
at $p$ from one side. 
We obtain a contradiction by bounding the Euclidean distance 
$\rho(L_k,L_{k+1})$ from below, for infinitely many $k$.

Here is the detailed construction. 
Let $\nu$ be the line in $P$ parallel to the intersection of $P$ with the
stem of $\partial\HH_0$. Then $\nu$ makes an angle of at most $\pi/4$ with
every ray contained in $\HH_0\cap P$.
Let $L_k\subset P$ be the line perpendicular to $\nu$ bounding a half-plane
$\Pi_k$ containing $\HH_k\cap P$ and intersecting the zigzag
$\zeta_k=\partial\HH_k\cap P$ at a vertex
of $\zeta_k$.

Choose $\delta_0>0$ as in Lemma~\ref{lem:delta0}. 

\begin{lem}\label{lem:tubnbhd}
For any $\delta\le\delta_0$, the tubular neighborhood 
\begin{equation}\label{eq:tubnbhddef}
T_k(\delta) = \gamma_k\left(B(\gamma_k^{-1} L_k,\delta)\right) 
\end{equation}
of $L_k$ is disjoint from $L_{k+1}$.
\end{lem}
\begin{proof}
$L_k\subset P-\HH_k$ implies  
\begin{equation}\label{eq:gammaL}
\gamma_k^{-1}L_k\subset \E -\Hh_{i_k}^{j_k}.  
\end{equation}
Now
\begin{align*}
\E -\Hh_{i_k}^{j_k} & = \E - \gamma_k^{-1}\HH_k \\ 
& = 
\gamma_k^{-1}(\E - \HH_k) \\ 
& \subset 
\gamma_k^{-1}(\E - \HH_{k+1}) \\ 
& = 
\E - \gamma_k^{-1}\gamma_{k+1} \Hh_{i_{k+1}}^{j_{k+1}} \\ 
& = 
\E - h_{i_k}^{j_k} \Hh_{i_{k+1}}^{j_{k+1}}.
\end{align*}
Apply \eqref{eq:gammaL} and \eqref{eq:uniformdistance} to conclude
$B(\gamma_k^{-1},\delta) \subset \E - h_{i_k}^{j_k}
\bar{\Hh}_{i_{k+1}}^{j_{k+1}}$, so 
\begin{align*}
T_k(\delta) & =  \gamma_k B(\gamma_k^{-1}L_k,\delta) \\ 
& \subset 
\E - h_{i_0}^{j_0}h_{i_1}^{j_1}\dots h_{i_{k-1}}^{j_{k-1}} 
(h_{i_k}^{j_k} \bar{\Hh}_{i_{k+1}}^{j_{k+1}}) \\
& = E - \bar{\HH}_{k+1}.
\end{align*}
so  $T_k(\delta)$ is disjoint from $\bar{\HH}_{k+1}$.
Intersecting with $P$, $T_k(\delta)$ is disjoint from $L_{k+1}$.
\end{proof}
\subsection{Bounding the separation of half-planes}
Write $E^k(x)$ for the weak-unstable plane 
$E^{wu}_x(g_k)$ where $x\in\E$ and and $g_k$ is the linear part of $\gamma_k$
(see Definition~\ref{def:weakunstable}). Foliate $T_k (\delta)$ by
leaves $E^k(x)\cap T_k (\delta)$. We first bound the diameter of the
leaves of $T_k(\delta)$.

\begin{lem}\label{lem:wunu} 
The angle between $\nu$ and any
$E^k(x)\cap P$ is bounded by $\pi/4$.
\end{lem}
\begin{proof} 
By Lemma~\ref{lem:Brouwer}, the vector $\xp{\gamma_k}$ lies in
the attracting interval $A_{i_0}^{j_0}$. The null plane corresponding to
$\xp{g_{i_0}^{j_0}}$ is the weak-unstable plane $E^k(x)$.
Since the stem of the crooked
plane $\Cc_{i_0}^{j_0}$ corresponds to $A_{i_0}^{j_0}$, 
the corresponding null plane $E^k(x)$ intersects $\HH_0 = \Hh_{i_0}^{j_0}$ in
a half-plane.  (Compare Figure~\ref{fig:ray1}.)

% Thus $E^k(x)$ meets the crooked half-space
% $\HH_k\subset\HH_0$ in a half-plane. 
Hence the line $E^k(x)\cap P$ meets $\HH_0\cap P$ in a ray. 
Since any ray in $\HH_0\cap P$ subtends an angle of at most 
$\pi/4$ with $\nu$, Lemma~\ref{lem:wunu} follows.
\end{proof} 

\begin{figure}[ht]
\centerline{\epsfxsize=3in \epsfbox{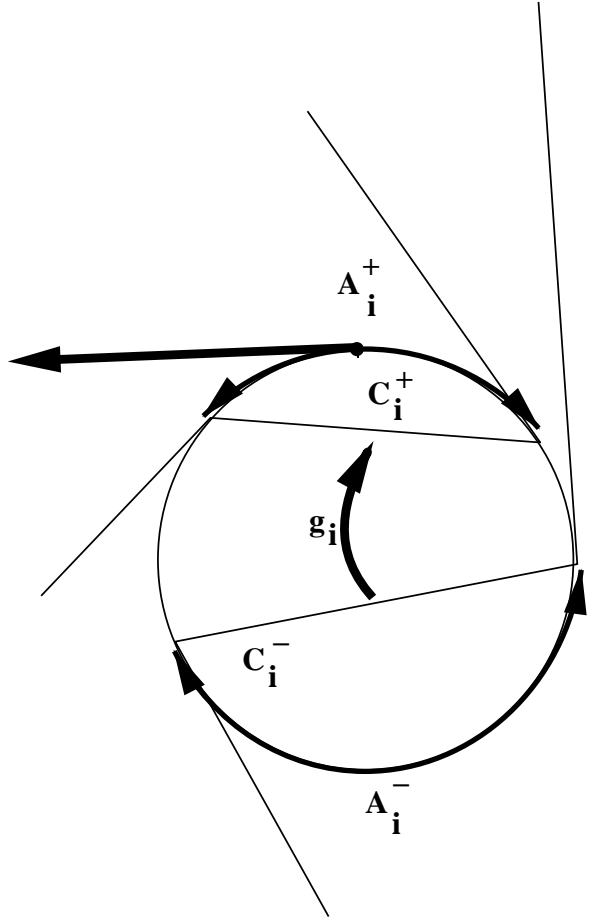}}
\caption{Weak-unstable ray in a crooked half-space }
\label{fig:ray1}
\end{figure}

\begin{lem}\label{lem:separation} 
Let $\epsilon>0$.
If $\gamma_k$ is $\epsilon$-hyperbolic and $\delta < \delta_0$, 
then
\begin{equation*}
\rho(L_k,L_{k+1}) \ge \frac{\delta\epsilon}{4\sqrt{2}}.
\end{equation*}
\end{lem}
\begin{proof}
Apply the Compression Lemma~\ref{lem:CompressionLemma} with 
$x\in\gamma_k^{-1}(L_k)$ and $h=\gamma_k$ to obtain
\begin{equation*}
B \Big( 
L_k,\frac{\delta\epsilon}4 
\Big) 
\cap E^k(x) \subset 
\gamma_kB(\gamma_k^{-1}L_k,\delta) = T_k(\delta).
\end{equation*}
Lemma~\ref{lem:tubnbhd} implies that the tubular neighborhood 
$T_k (\delta)$ is disjoint from $L_{k+1}$.
Therefore
$B(L_k,\delta\epsilon/4) \cap E^k(x)$ is disjoint from $L_{k+1}$ and 
\begin{equation}\label{eq:1}
\rho\left(x,\partial T_k (\delta)\right) 
\le  \rho(x,L_{k+1}) \le \rho(L_k,L_{k+1}).
\end{equation}
%% old code
% \begin{equation}\label{eq:1}
% \rho\left(L_k,B\Big(L_k,\frac{\delta\epsilon}4\Big)\cap E^k(x)\right) \le 
% \rho(L_k,L_{k+1}).
% \end{equation}
%% old code ends here
Lemma~\ref{lem:wunu} implies $\angle\left(\nu,E^k(x)\cap P\right) < \pi/4$ so 
\begin{equation*}
\cos\angle\left(\nu,E^k(x)\cap P\right) > \frac{1}{\sqrt{2}}.
\end{equation*}
Thus (compare Figure~\ref{fig:tubnbhd})
\begin{equation}\label{eq:2}
\rho \left(x,\partial T_k (\delta)\right) = \rho\left(x,\partial B\Big(L_k,\frac{\delta\epsilon}4\Big)\cap E^k(x)\right)\cos\angle\left(\nu,E^k(x)\cap P\right) > \frac{\delta\epsilon}{4\sqrt{2}}.
\end{equation}
% \begin{equation}\label{eq:2}
% \rho\left(L_k,B\Big(L_k,\frac{\delta\epsilon}4\Big)\cap E^k(x)\right) =
% \cos\angle\left(\nu,E_k(x)\cap P\right)
% \frac{\delta\epsilon}4 > \frac{\delta\epsilon}{4\sqrt{2}}.
% \end{equation}
Lemma~\ref{lem:separation} follows from \eqref{eq:1} and \eqref{eq:2}.
\end{proof}
\begin{figure}[htb]
\centerline{\epsfxsize=3in \epsfbox{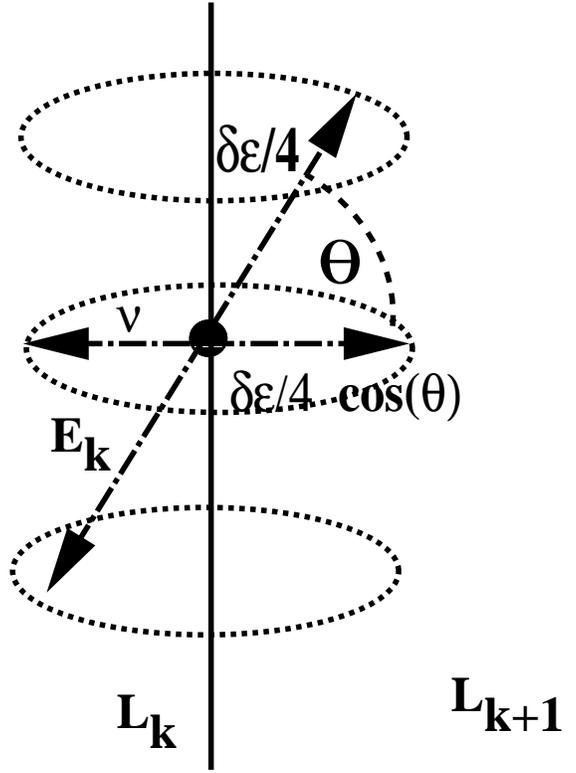}}
\caption
{Separating the linear approximations}
\label{fig:tubnbhd}
\end{figure}

\subsection{The alternative to $\epsilon$-hyperbolicity}
\label{sec:alternative}

By Lemma~\ref{lem:separation}, it suffices to find $\epsilon>0$ such that
for infinitely many $k$, the element $\gamma_k$ is $\epsilon$-hyperbolic.
Lemma~\ref{lem:fxpts} gives a criterion for $\epsilon$-hyperbolicity
in terms of the expression of $\gamma_k$ as a reduced word.

Choose $\epsilon_0$ as in Lemma~\ref{lem:fxpts} and
the sequence $(i_0,j_0),\dots, (i_k,j_k),\dots$ as in 
Lemma~\ref{lem:nestedsequence}. 
Recall that $\gamma_k$ has the expression
\begin{equation*}
\gamma_k  = h_{i_0}^{j_0}\dots h_{i_{k-1}}^{j_{k-1}}
\end{equation*}
where 
$(i_{k+1},j_{k+1})\neq (i_k,-j_k)$ for all $k\ge 0$.
Lemma~\ref{lem:fxpts} implies:

\begin{lem}\label{lem:notehyp}
Either $g_k$ is $\epsilon_0$-hyperbolic for infinitely many $k$ or
there exists $k_2>0$ such that $(i_k,j_k)= (i_0,-j_0)$ for all $k>k_2$.
We may assume that $k_2$ is minimal, that is, 
$(i_{k_2},j_{k_2})\neq (i_0,-j_0)$.
\end{lem}

Thus $g_k$ has the special form
\begin{equation*}
g_k =  (g_{i_0}^{j_0})^{k_1} g' (g_{i_0}^{-j_0})^{k - k_1 - k_2 - 1} 
\end{equation*}
where $k_1>0$ is the smallest $k$ such that $(i_k,j_k)\neq(i_0,j_0)$ and
$g'$ is the subword
\begin{equation*}
g_{i_{k_1}}^{j_{k_1}} \dots g_{i_{k_2}}^{j_{k_2}}.
\end{equation*}
$g'$ is the maximal subword of $g_k$ which % doesn't begin 
neither begins
%correction
with $g_{i_0}^{j_0}$
% and 
nor end 
%correction:	I think ``neither ... nor'' sounds better than ``doesn't''
ends with $g_{i_0}^{-j_0}$. In particular the conjugate of $g_k$
by $\psi= g_{i_0}^{-j_0k_1}$
\begin{equation*}
g_{i_0}^{-j_0k_1} g_k g_{i_0}^{j_0k_1} = g' (g_{i_0}^{-j_0})^{k - k_2 - 1} 
\end{equation*}
is $\epsilon_0$-hyperbolic, by Lemma~\ref{lem:fxpts}.

\subsection{Changing the hyperbolicity}\label{sec:change}
The proof concludes by showing that there is a $K>1$ depending
on $g_{i_0}^{-j_0k_1}$ and taking $\epsilon$ smaller than
$\epsilon_0/K$, infinitely many $g_k$ are $\epsilon$-hyperbolic for 
this new choice of $\epsilon$. 
This contradiction concludes the proof of the theorem.

\begin{lem}\label{lem:conjhyp}
Let $\psi\in\SOoo$.
Then there exists $K$ such that, for any $\epsilon>0$,
an element $g\in\SOoo$ 
is $\epsilon/K$-hyperbolic whenever $\psi g\psi^{-1}$ is 
$\epsilon$-hyperbolic. 
\end{lem}
\begin{proof}
Let $s$ denote the distance $d(O,\psi(O))$ that $s$ moves the origin 
$O\in\rht$ (see \eqref{eq:rhtorigin}) and let 
\begin{equation*}
K  = e^s \pi/2.  
\end{equation*}

Since $\xpm{\psi g\psi^{-1}} = \psi\left(\xpm{g}\right)$,
it suffices to prove that if 
$\va_1,\va_2\in S^1$, then
\begin{equation}\label{eq:distortion}
K^{-1} \le \frac{\rho(\psi(\va_1),\psi(\va_2))}{\rho(\va_1,\va_2)}  \le K.
\end{equation}

Let $\SOt$ be the group of rotations and $\SOooo$ the group of
{\em transvections\/}
\begin{equation*}
\tau_s = 
\bmatrix 
1 & 0  & 0  \\
0 & \cosh(s) & \sinh(s) \\
0 & \sinh(s) & \cosh(s) \endbmatrix.
\end{equation*}
Since $\SOt\subset\SOoo$ is a maximal compact subgroup and $\SOooo\subset\SOoo$
is an $\R$-split Cartan subgroup, the Cartan decomposition of $\SOoo$
is
\begin{equation*}
\SOoo = \SOt \cdot \SOooo \cdot \SOt
\end{equation*}
and we write $\psi = R_\theta \tau_s R_{\theta'}$
where $s = d(O,\psi(O))$ as above. Since 
\begin{equation*}
\rho(\psi(\va_1),\psi(\va_2)) = 
\rho(
\tau_s(R_{\theta'}(\va_1)),
\tau_s(R_{\theta'}(\va_2))) 
\end{equation*}
and 
\begin{equation*}
\rho(\va_1,\va_2) = 
\rho(
R_{\theta'}(\va_1),
R_{\theta'}(\va_2)),
\end{equation*}
it suffices to prove \eqref{eq:distortion} for $\psi = \tau_s$.

In this case 
\begin{equation*}
\frac{d\phi}{\psi^*d\phi} = \frac{1+\cos{\phi}}{2}e^{s} +  
\frac{1-\cos{\phi}}{2}e^{-s}
\end{equation*}
so that 
\begin{equation*}
e^{-s} \le \frac{d\phi}{\psi^*d\phi} \le e^{s}.
\end{equation*}
Let $A$ be the interval 
on $S^1$ joining $\va_1$ to $\va_2$, such that $\Phi(A)\le\pi$. Its length and the length
of its image $\psi(A)$ are given by:
\begin{equation*}
\Phi(A) = \int_A \vert d\phi\vert, \qquad \Phi(\psi(A)) = \int_{\psi(A)} \vert
d\phi\vert
= \int_A \psi^*(\vert d\phi\vert).
\end{equation*}
Therefore
\begin{equation}\label{eq:Riemdistortion}
e^{-s} \le  \frac{\Phi(A)}{\Phi(\psi(A))} \le e^{s}. 
\end{equation}
Finally the distance $\rho(\va_1,\va_2)$ on $S^1$ (the length of the chord
jointing $\va_1$ to $\va_2$)
relates to the Riemannian distance by 
$\rho(\va_1,\va_2) = 2 \sin (\Phi(A)/2). 	$
Now (since $-\pi \le \phi \le \pi$)
\begin{equation*}
\frac2\pi \le \frac{2\sin(\phi/2)}\phi \le 1,
\end{equation*}
implies
\begin{equation}\label{eq:RiemChord1}
\frac2\pi \le\frac{\rho(\va_1,\va_2)}{\Phi(A)}
\le 1
\end{equation}
and
\begin{equation}\label{eq:RiemChord2}
%\frac2\pi \le\frac{\rho(\psi(\va_1),\psi(\va_2))}{\Phi(\psi(A))} \le 1.
1 \le \frac{\Phi(\psi(A))}{\rho(\psi(\va_1),\psi(\va_2))} \le \frac{\pi}2.
\end{equation}
Combining inequalities
\eqref{eq:Riemdistortion} with \eqref{eq:RiemChord1} and
\eqref{eq:RiemChord2} implies \eqref{eq:distortion}.
\end{proof}

%%%%%%%%%%%%%%%%%%%%%%%%%%%%%%%%%%%%%%%%%%%%%%%%%%%%%%%%%%%%%%%%%%%%%%%%%%%%
\makeatletter \renewcommand{\@biblabel}[1]{\hfill#1.}\makeatother

\begin{figure}[ht]
\centerline{\epsfxsize=3in \epsfbox{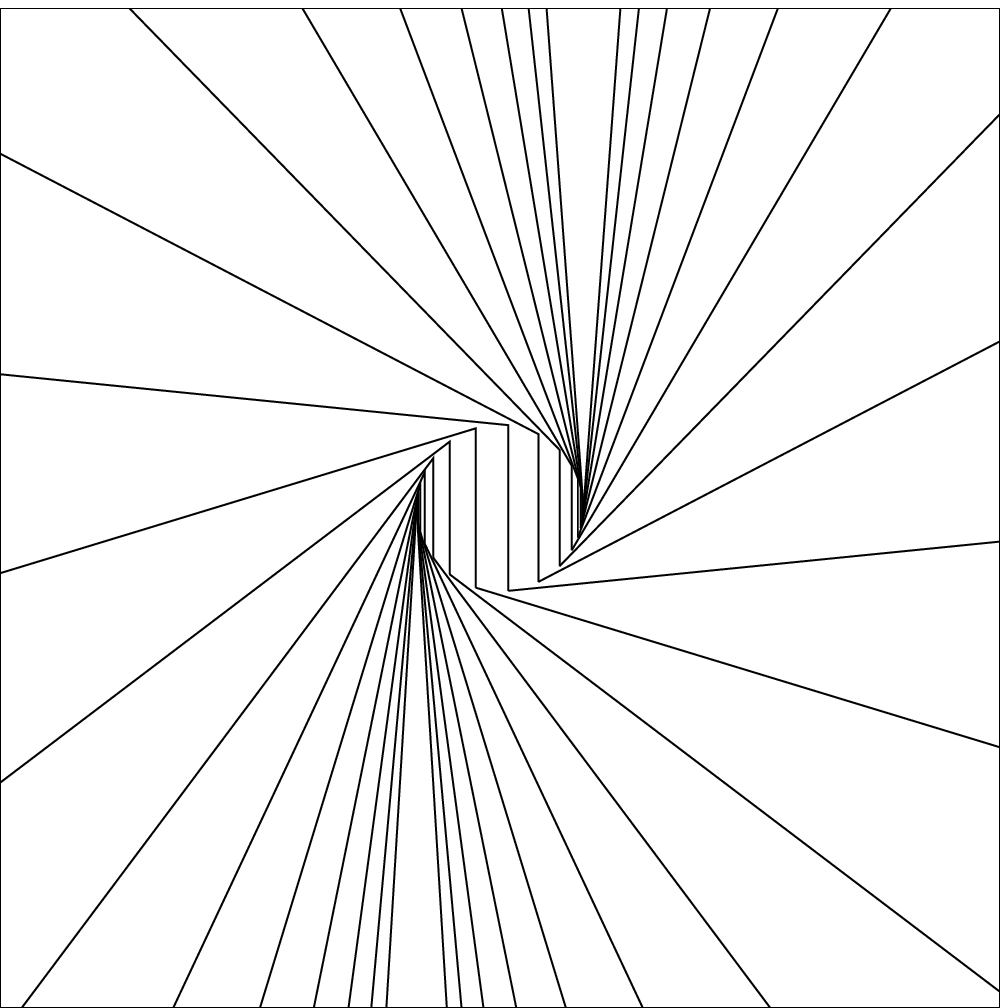}}
\caption
{Zigzags for a linear hyperbolic cyclic group}
\label{fig:lcycliczz}
\end{figure}

\begin{figure}[ht]
\centerline{\epsfxsize=3in \epsfbox{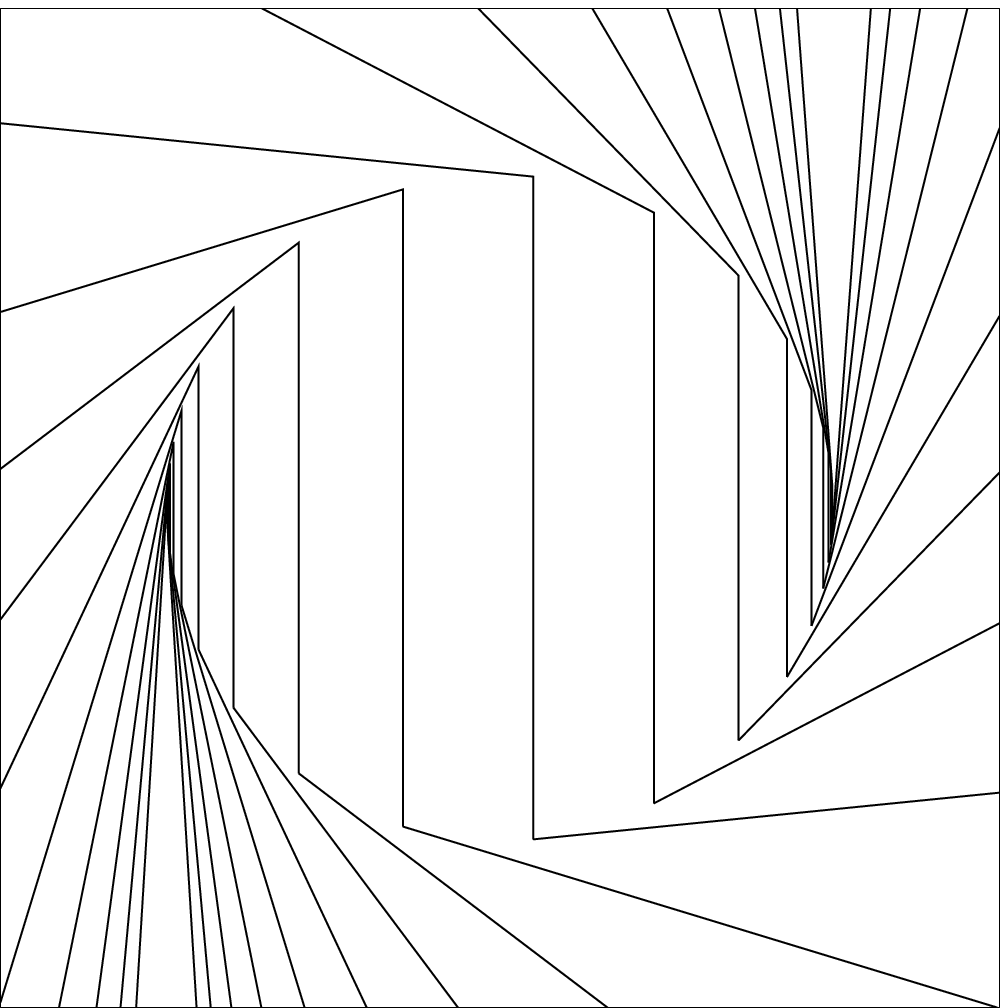}}
\caption
{Zigzags for a linear hyperbolic cyclic group:close-up}
\label{fig:biglin}
\end{figure}

\begin{figure}[ht]
\centerline{\epsfxsize=3in \epsfbox{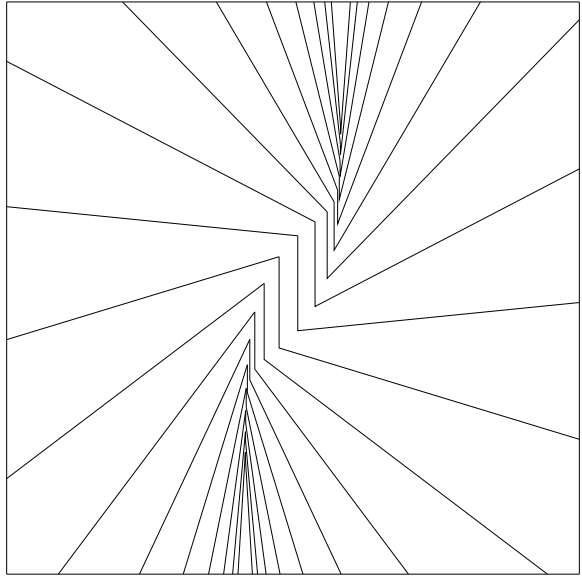}}
\caption
{Zigzags for an affine hyperbolic cyclic group}
\label{fig:acycliczz}
\end{figure}

\begin{figure}[ht]
\centerline{\epsfxsize=3in \epsfbox{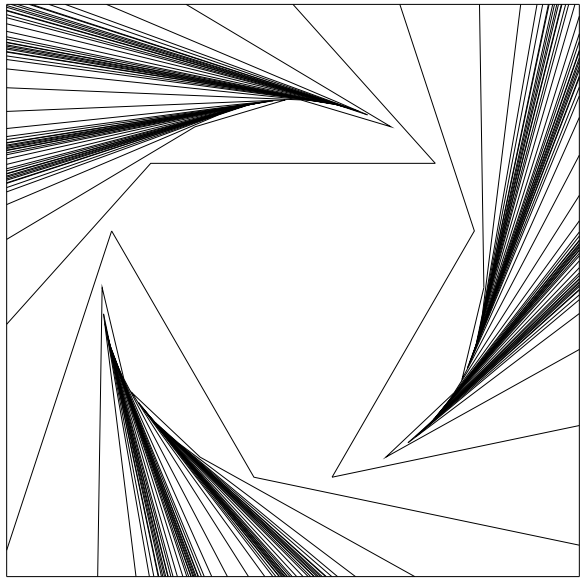}}
\caption
{Tiling by a discrete linear group}
\label{fig:linearzzs}
\end{figure}

\begin{figure}[ht]
\centerline{\epsfxsize=3in \epsfbox{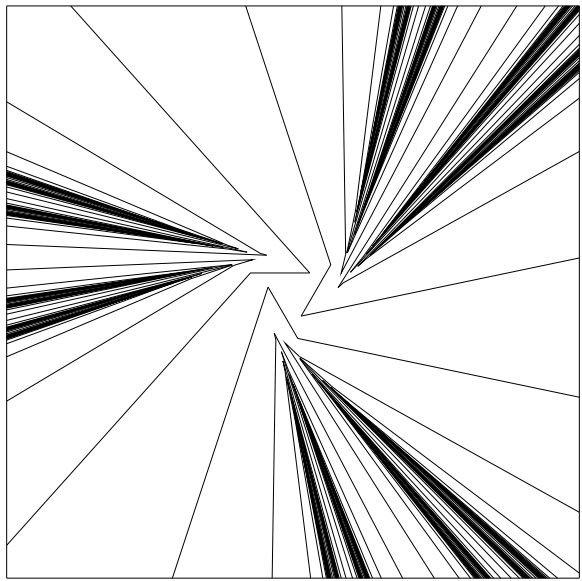}}
\caption
{Zigzags for an ultra-ideal triangle group: Remote view}
\label{fig:ultraideal1}
\end{figure}

\begin{figure}[ht]  
\centerline{\epsfxsize=3in \epsfbox{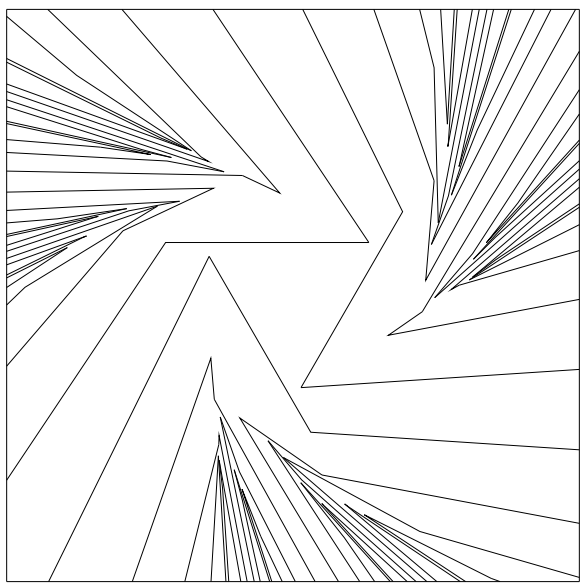}}
\caption
{Zigzags for an ultra-ideal triangle group: Close-up view}
\label{fig:ultraideal2}
\end{figure}
%\listoffigures
\end{document}